\documentclass[12pt, a4paper]{amsart}
\usepackage{amssymb}
\usepackage{amsmath,esint,enumitem}
\usepackage{colortbl}
\usepackage[dvipsnames]{xcolor}
\usepackage[colorlinks, allcolors=Fuchsia,pdfstartview=,pdfpagemode=UseNone]{hyperref} 
\usepackage{graphics,fullpage}
\usepackage{multirow}
\makeatletter
\@namedef{subjclassname@2020}{\textup{2020} Mathematics Subject Classification}
\makeatother

\newtheorem{theorem}[equation]{Theorem}
\newtheorem{lemma}[equation]{Lemma}
\newtheorem{proposition}[equation]{Proposition}

\theoremstyle{definition}
\newtheorem{definition}[equation]{Definition}

\newtheorem{example}[equation]{Example}

\theoremstyle{remark}
\newtheorem{remark}[equation]{Remark}
\numberwithin{equation}{section}

\let\oldmarginpar\marginpar
\renewcommand\marginpar[1]{\-\oldmarginpar[\raggedleft\footnotesize #1]%
{\raggedright\footnotesize #1}}

\newcommand{\supp}{\operatorname{supp}}
\DeclareMathOperator*{\osc}{osc}
\newcommand{ \N }{ \mathbb{N} }
\newcommand{ \R }{ \mathbb{R} }
\newcommand{ \Rn }{ {\mathbb{R}^n} }

\newcommand{\Phiw}{{\Phi_{\mathrm{w}}}}
\newcommand{\Phic}{{\Phi_{\mathrm{c}}}}

\newcommand{\tphi}{{\tilde\phi}}

\newcommand{\tpsi}{{\tilde\psi}}
\newcommand{\ttheta}{{\tilde\theta}}

\newcommand{\tomega}{\tilde\omega}

\newcommand{\tA}{{\tilde A}}
\newcommand{\tu}{{\tilde u}}

\renewcommand{\epsilon}{\varepsilon}
\renewcommand{\phi}{\varphi}
\renewcommand{\le}{\leqslant}
\renewcommand{\ge}{\geqslant}
\renewcommand{\leq}{\leqslant}
\renewcommand{\geq}{\geqslant}
\renewcommand{\div}{\mathrm{div}}
\newcommand{\loc}{\mathrm{loc}}
\newcommand{\BMO}{\mathrm{BMO}}

\newcommand{\ainc}[1]{\hyperref[ainc]{{\normalfont(aInc){\ensuremath{_{#1}}}}}}
\newcommand{\adec}[1]{\hyperref[adec]{{\normalfont(aDec){\ensuremath{_{#1}}}}}}
\newcommand{\inc}[1]{\hyperref[inc]{{\normalfont(Inc){\ensuremath{_{#1}}}}}}
\newcommand{\dec}[1]{\hyperref[dec]{{\normalfont(Dec){\ensuremath{_{#1}}}}}}
\newcommand{\azero}{\hyperref[azero]{{\normalfont(A0)}}}
\newcommand{\aone}{\hyperref[aone]{{\normalfont(A1)}}}

\newcommand{\wVA}{\hyperref[wVA]{{\normalfont(wVA1)}}}
\newcommand{\VA}{\hyperref[VA]{{\normalfont(VA1)}}}

\newcommand{\VMA}{\hyperref[VMA]{{\normalfont(VMA1)}}}

\newcommand{\DMA}[1]{\hyperref[DMA]{{\normalfont(DMA1)\ensuremath{_{#1}}}}}
\newcommand{\SA}{\hyperref[SA]{{\normalfont(SA1)}}}
\newcommand{\aonen}[1]{\hyperref[aone]{{\normalfont(A1-\ensuremath{#1})}}}
\newcommand{\VAn}[1]{\hyperref[VA]{{\normalfont(VA1-\ensuremath{#1})}}}
\newcommand{\wVAn}[1]{\hyperref[wVA]{{\normalfont(wVA1-\ensuremath{#1})}}}
\newcommand{\VMAn}[1]{\hyperref[VMA]{{\normalfont(VMA1-\ensuremath{#1})}}}
\newcommand{\wVMAn}[1]{\hyperref[wVMA]{{\normalfont(wVMA1-\ensuremath{#1})}}}
\newcommand{\DMAn}[1]{\hyperref[DMA]{{\normalfont(DMA1-\ensuremath{#1})}}}
\newcommand{\wDMAn}[1]{\hyperref[wDMA]{{\normalfont(wDMA1-\ensuremath{#1})}}}
\newcommand{\SAn}[1]{\hyperref[SA]{{\normalfont(SA1-\ensuremath{#1})}}}

\begin{document}

\title{Mean oscillation conditions for nonlinear equation and regularity results}

\author{Peter H\"ast\"o}
\address{Department of Mathematics and Statistics, FI-00014 University of Helsinki, Finland
and 
Department of Mathematics and Statistics, FI-20014 University of Turku, Finland}
\email{peter.hasto@helsinki.fi}

\author{Mikyoung Lee}
\address{Department of Mathematics and Institute of Mathematical Science, Pusan National University, Busan 46241, Republic of Korea}
\email{mikyounglee@pusan.ac.kr}

\author{Jihoon Ok}
\address{Department of Mathematics, Sogang University, Seoul 04107, Republic of Korea}
\email{jihoonok@sogang.ac.kr}

\thanks{}

\date{\today}

\begin{abstract}
We consider general nonlinear elliptic equations of the form 
\[ 
\operatorname{div} A(x,Du) = 0 \quad \text{in } \Omega, 
\] 
where $ A:\Omega \times \Rn \to \Rn $ satisfies a quasi-isotropic $(p,q)$-growth condition, which is equivalent to the point-wise uniform ellipticity of $A$. We establish sharp and comprehensive mean oscillation conditions on $A(x,\xi)$ with respect to the $x$ variable to obtain $C^1$- and $W^{1,s}$-regularity results. The results provide new conditions even in the standard $p$-growth case 
with coefficient $\operatorname{div}(a(x)|Du|^{p-2}Du)=0$. 
Also covered are non-standard growth cases such as variable exponent and double-phase. 
Furthermore, we include Sobolev--Lorentz-type conditions as part of our unified framework.
\end{abstract}

\maketitle


\section{Introduction}\label{sect:intro}

We study regularity properties of weak solutions to the following nonlinear, non-autonomous equation
\begin{equation}\tag{$\div A$}\label{maineq}
\mathrm{div} A(x,Du) =0
\end{equation}
in a domain $\Omega\subset \R^n$, $n\ge 2$, where the nonlinearity 
$A:\Omega\times \R^n\to \R^n$ satisfies the quasi-istropic $(p,q)$-growth condition in Definition~\ref{def:quasiisotropy}. In particular, we are interested in sharp 
conditions on $A$ of mean oscillation type that imply desired 
regularity of the weak solution to \eqref{maineq} and cover all 
different types of assumptions of earlier results.
Our results are new even in the standard $p$-growth case 
with coefficient $\operatorname{div}(a(x)|Du|^{p-2}Du)=0$ and allow us to treat previously 
disparate Sobolev--Lorentz-type conditions, among others, 
as part of a unified framework.

Let us start with reviewing relevant previous results. 
Regularity is well-known for linear equations with coefficient, 
$A(x,\xi) = M(x)\xi$, i.e. 
\begin{equation}\tag{$\Delta_{2,M}$}\label{eq:linear}
\mathrm {div}(M(x)Du)=0 \quad \text{in }\ \Omega,
\end{equation}
where the coefficient matrix $M: \Omega\to \mathbb M_n$ satisfies $L^{-1}\le M(x)e \cdot e \le L$ for 
some $L\ge 1$ and all $x \in \Omega$ and $e\in \partial B_1(0)$. 
We consider three implications in various settings:
\begin{align}
\tag{$\Rightarrow C^{1,\alpha}$}\label{eq:holder-C} 
\omega(r)\lesssim r^\beta &\quad \Longrightarrow \quad u\in C^{1,\alpha}_{\loc}(\Omega),\\
\tag{$\Rightarrow W^{1,s}$}\label{eq:vanishing-C}  
\lim_{r\to 0^+} \omega(r)=0 &\quad \Longrightarrow \quad u\in W^{1,s}_{\loc}(\Omega) \ \ \forall s>1,\\
\tag{$\Rightarrow C^1$}\label{eq:Dini-C} 
\int_0^1 \omega (r)\frac{dr}{r}<\infty &\quad \Longrightarrow \quad u \in C^1_\loc(\Omega),
\end{align}
where $\omega$ is chosen differently according to the setting. 
%
Note that \eqref{eq:vanishing-C} includes the result $u\in C^{0,\alpha}$ for any $\alpha\in (0, 1)$ by the Sobolev embedding, 
as well as Calderón--Zygmund type gradient estimates for the corresponding nonhomogeneous equations, as stated in Theorem~\ref{thm:CZ}.
In the linear setting with $\omega$ equal to the point-wise continuity modulus 
\[
\omega_M(r):= \sup_{B_r\subset \Omega}\sup_{x,y\in B_r} |M(x)-M(y)|
\]
the implications \eqref{eq:holder-C}, \eqref{eq:vanishing-C} and \eqref{eq:Dini-C} 
were proved by Caccioppoli/Schauder \cite{Caccio34, Scha34, Scha37}, Agmon--Douglis--Nirenberg 
\cite{ADN59} and Hartman--Wintner \cite{HW55}, respectively. 

More recently, Di Fazio \cite{DiFazio96} and Dong--Kim \cite{DongKim17} showed, respectively, 
that \eqref{eq:vanishing-C} and \eqref{eq:Dini-C} remain true under the weaker assumption that 
$\omega$ is 
the modulus of mean continuity
\[
\bar\omega_M(r):= \sup_{B_r\subset \Omega}\fint_{B_r}\fint_{B_r} |M(x)-M(y)|\, dx\,dy.
\]
Note that the implication \eqref{eq:holder-C} is trivial for $\bar\omega_M$ since 
this modulus has an upper bound $r^\beta$ if and only if $\omega_M$ does, by the Campanato embedding. 
We refer to \cite{Ancona09,Brezis08,DongKim18,JMV09,Li17} for more related results and the sharpness of  
mean oscillation type conditions. In short, mean oscillation type conditions are essentially the optimal ones in regularity theory for partial differential equations.

Many of the above implications have been generalized to the $p$-Laplace equations with coefficients of the form
\begin{equation}\tag{{$\Delta_{p,M}$}}\label{eq:pLaplace}
\mathrm {div}\left((M(x)Du\cdot Du)^{\frac{p-2}{2}}M(x)Du\right)=0 \quad \text{in }\ \Omega, 
\quad 1<p<\infty,
\end{equation}
i.e. $A(x,\xi) = (M(x)\xi \cdot \xi)^{\frac{p-2}{2}}M(x)\xi$. 
Manfredi first proved \eqref{eq:holder-C} in \cite{Man88}.
Kuusi--Mingione \cite{KuuMin14} established the implication \eqref{eq:Dini-C} 
with the point-wise continuity modulus $\omega=\omega_M$, whereas  
Kinnunen--Zhou \cite{KinZhou99} established the implication \eqref{eq:vanishing-C} 
with the mean continuity modulus $\omega=\bar\omega_M$.
However, the implication \eqref{eq:Dini-C} with mean continuity modulus $\omega=\bar\omega_M$ has remained open for the $p$-Laplacian, even for scalar weights $M(x)=a(x) I_n$, where $I_n$ is the 
identity matrix.

Research into equations with nonstandard growth conditions has exploded in the past 25 years. 
Nonstandard growth means that the growth of the equation strongly depends on the $x$ variable, 
so that the equation does not satisfy the stronger global version of uniform ellipticity 
defined below.
A model equation with nonstandard growth is the $p(x)$-Laplace equation
\begin{equation}\tag{{$\Delta_{p(\cdot)}$}}\label{eq:pxLaplace}
\mathrm{div} \left(|Du|^{p(x)-2}Du\right)=0
\quad \text{in }\ \Omega, \quad 1<p_1\le p(x)\le p_2.
\end{equation}
Here, 
\eqref{eq:holder-C} with the point-wise modulus of continuity $\omega=\omega_p$ of $p$ was first proved by Coscia--Mingione \cite{CoMin99} and Acerbi--Mingione \cite{AceMin01}.  With 
$\omega(r)=\omega_p(r)\log \frac1r$ now including an extra logarithm,
the implications \eqref{eq:vanishing-C} and \eqref{eq:Dini-C} were proved in \cite{AceMin05,BOR16} and \cite{Ok16}, respectively. 
However, to the best of our knowledge no results have been proved under mean continuity conditions 
for the $p(\cdot)$-Laplacian. 
Similarly, for the borderline double phase problem
\begin{equation}\tag{{$\Delta_{\text{bdp}}$}}\label{eq:bdp}
\mathrm{div} \left(|Du|^{p-2}Du +a(x) \log (1+|Du|) |Du|^{p-2}Du \right) =0 
\quad \text{in }\ \Omega, 
\end{equation}
where $1<p<\infty$ and $0\le a(x) \le L$, 
regularity results have been obtained under point-wise continuity assumptions \cite{BarC22,BarCM16,ByunOh17}
but no results are available when assuming only mean continuity. 

In recent years, $C^{0,1}$- and $C^1$-regularity results under Sobolev--Lorentz type conditions have been developed. De Filippis and Mingione \cite{DeFilMin21} proved that if 
$M(x)=a(x)I_n$, where
$a\in W^{1,1}(\Omega)$ and $|Da|$ is in the Lorentz space $L^{n,1}(\Omega)$, then the weak solution to \eqref{eq:pLaplace} is locally Lipschitz.
In fact, they considered a more general class of nonuniformly elliptic equations.
Baroni and Coscia \cite{Bar23, Bar25, BarC22} have obtained $C^1$-regularity results for 
the \eqref{eq:bdp} and \eqref{eq:pxLaplace} under Lorenz-type assumptions 
on the coefficient and exponent. 
One of the advantages of our general framework is that also assumptions of Sobolev--Lorentz type 
are included as special cases of our Dini mean continuity condition, see Proposition~\ref{prop:Lorentz}.


\begin{table}\label{table:summary}
\begin{tabular}{lcccccc}
Assumption type&Conclusion&\eqref{eq:linear} & \eqref{eq:pLaplace} & \eqref{eq:pxLaplace} &
\eqref{eq:bdp}& \eqref{maineq} \\[6pt]
\multirow{3}{*}{Point-wise} 
&\eqref{eq:holder-C}& \cite{Caccio34, Scha34, Scha37} & \cite{Man88} & \cite{AceMin01,CoMin99} &\cite{BarCM16}&\cite{HasO22, HasO22b}\\
&\eqref{eq:vanishing-C}& \cite{ADN59}\phantom{*} &\cite{KinZhou99}& \cite{AceMin05,BOR16} & 
\cite{ByunOh17} &new\\
& \eqref{eq:Dini-C}& \cite{HW55}\phantom{*} & \cite{KuuMin14}&\cite{Ok16}& \cite{BarC22} &new\\[8pt]
\multirow{2}{*}{Mean}&\eqref{eq:vanishing-C}& \cite{DiFazio96}\phantom{*} &\cite{KinZhou99}& new & new &new\\
& \eqref{eq:Dini-C} &\cite{DongKim17}* & new& new& new &new\\[6pt]
Lorenz & $\Rightarrow C^1$ &\cite{DeFilMin21}\phantom{*}&\cite{DeFilMin21}&\cite{Bar23}& \cite{Bar25} & new\\[8pt]
\end{tabular}
\caption{A summary of previous results for different equations, assumptions and conclusions. 
The asterisk is explained in Remark~\ref{rem:gamma}.}
\end{table}

The results mentioned so far are summarized in Table~\ref{table:summary}. In the 
table we see that results with mean continuity assumptions are especially lacking, 
although also other categories have gaps for more complicated types of examples. 
In particular, general equations of type \eqref{maineq} have almost no 
higher regularity results and neither do special cases not mentioned in the 
table like variable exponent double phase energies (e.g., \cite{OhnS25, WeiXY25}).

In this paper, we fill in the blanks in the table by proving results with mean continuity assumptions 
for a general equation with the following fundamental growth and ellipticity conditions on the nonlinearity $A(x,\xi)$. For \ainc{} and \adec{}, see Definition~\ref{def:ainc}.

\begin{definition}\label{def:quasiisotropy}
Let $1<p<q$. We say that $A:\Omega\times \R^n\to \R^n$ with 
$A(x,\cdot)\in C^1(\R^n\setminus\{0\},\R^n)$ satisfies \textit{quasi-isotropic $(p,q)$-growth} condition if
\begin{itemize}
\item[(i)] $D_\xi A(x,\xi)$ satisfies \ainc{p-2} and \adec{q-2} with constant $L\ge 1$ ($(p,q)$-growth); 
\item[(ii)] for every $x\in\Omega$ and $e,\xi,\xi'\in\R^n$ with $|e|=1$ and $|\xi|=|\xi'|>0$,
\[
|D_\xi A(x,\xi')| \le L\, D_\xi A(x,\xi)e \cdot e
\] 
for some $L\ge 1$ (quasi-isotropic ellipticity).
\end{itemize}
\end{definition}

When $A(x,\xi)=D_\xi f(x,\xi)$ for some $f(x,\cdot)\in C^1(\R^n)\cap C^2(\R^n\setminus\{0\})$, the condition (ii) is equivalent to the following \textit{point-wise uniform ellipticity} condition:
\[
\mathcal R (x,t) := \frac{\sup\{\text{eigenvalues of }D_\xi A(x,\xi) \,:\, |\xi|=t \}}{\inf\{\text{eigenvalues of }D_\xi A(x,\xi)\, :\, |\xi|=t \}}\le L \quad \text{for all }\ x\in \Omega\ \ \text{and}\ \ t>0,
\]
which plays an important role in regularity theory of partial differential equations and the calculus of variations.  In particular, sharp $C^{1,\alpha}$-regularity under oscillation type condition on $A$ has been obtained in \cite{HasO22,HasO22b,HasO23}.
Note that the stronger global version of uniform ellipticity condition:
\[
\mathcal R_\Omega (t) := \frac{\sup\{\text{eigenvalues of }D_\xi A(x,\xi) \,:\, x\in\Omega,\ |\xi|=t \}}{\inf\{\text{eigenvalues of }D_\xi A(x,\xi)\, :\, x\in\Omega,\ |\xi|=t \}}\le L \quad \text{for all }\ t>0
\]
is significant as well. If the nonlinearity $A$ does not satisfy the global uniform ellipticity, 
an additional condition is essential to obtain even basic regularity results like H\"older continuity 
and higher integrability.
This situation is often referred to as the nonstandard growth case and occurs for instance 
in the $p(x)$-Laplace equation \eqref{eq:pxLaplace} and the double phase problem \eqref{eq:dp}.
We also refer to recent development in regularity theory for nonuniformly elliptic equations by De Filippis and Mingione \cite{DeFilMin21,DeFilMin23,DeFilMin231,DeFilMin25}. 

Another important model equation is the double phase equation
\begin{equation}\tag{{$\Delta_{\text{dp}}$}}\label{eq:dp}
\mathrm{div} \left(|Du|^{p-2}Du +a(x)  |Du|^{q-2}Du \right) =0 
\quad \text{in }\ \Omega, 
\end{equation}
where $1<p\le q <\infty$ and $0\le a(x) \le L$. For this equation,  $C^{1,\alpha}$- and $W^{1,s}$-regularity results have been established in \cite{BarColMin18,ColMin15,ColMin16,DeFilMin20} by assuming point-wise continuity $a\in C^{0,\alpha}(\Omega)$ and the inequality $\frac qp \le 1+ \frac{\alpha}{n}$. Note that this energy is not included in the Table~\ref{table:summary}, since there is 
no gap for mean continuity results due to the Campanato embedding, see Example~\ref{eg:double}.

Now, we state our main results, Theorems~\ref{thm:CZ} and \ref{thm:C1}, which provide sharp and comprehensive $C^1$- and $W^{1,s}$-regularity results for (point-wise) uniformly elliptic equations. Moreover, we will see in Section~\ref{sect:examples} that the theorems imply new results even in the special cases of Table~\ref{table:summary}. Define 
\[
A^{(-1)}(x,\xi):= |\xi| A(x,\xi).
\]
The conditions on $A$ in the following theorems are explained in Definitions~\ref{def:continuity} and 
\ref{def:meanoscillation}. 

We first consider a nonhomogeneous 
version of \eqref{maineq} in the context of Calder\'on--Zygmund type estimates:
\begin{equation}\tag{{$\div A; F$}}\label{maineq1}
\mathrm{div} A(x,Du) = \mathrm{div}A(x,F) 
\quad \text{in }\ \Omega,
\end{equation}
where the given function $F\in L^1 (\Omega)$ satisfies $ |A^{(-1)}(\cdot, F)| \in L^1_\loc(\Omega)$. 
We say that a function $u\in W^{1,1}_\loc(\Omega)$ is a weak solution to \eqref{maineq1} if $|A^{(-1)}(\cdot, Du)|\in L^1_\loc (\Omega)$ and 
\[
\int_{\Omega} A(x,Du)\cdot Dv \, dx = \int_{\Omega} A(x,F)\cdot Dv \, dx
\]
for all $v\in C^\infty_c(\Omega).$
Note that 
\eqref{maineq} is the special case when $F\equiv 0$, so the following result implies 
in particular the $W^{1,s}$-regularity for the equation \eqref{maineq}, 
i.e.\ \eqref{eq:vanishing-C}.

\begin{theorem}\label{thm:CZ}
Let $A:\Omega \times\R^n\to \R^n$ satisfy the quasi-isotropic $(p,q)$-growth condition, and let $u\in W^{1,1}_\loc(\Omega)$ with $|A^{(-1)}(\cdot, Du)|\in L^1_\loc (\Omega)$ be a weak solution to \eqref{maineq1}.
If $A^{(-1)}$ satisfies \VMA{} and \SA{} and $|A^{(-1)}(\cdot, F)| \in L^s_\loc(\Omega)$ for some $s>1$, then $|A^{(-1)}(\cdot, Du)| \in L^s_\loc(\Omega)$. Moreover, for any $\Omega'\Subset \Omega$ there exists a small $R_0>0$ depending on $n,p,q,L,s,\theta_1, \Omega'$ and $Du$ such that
\[
\fint_{B_r} |A^{(-1)}(x,Du)|^s\,dx \le c\left( \fint_{B_{2r}} |A^{(-1)}(x,Du)|\,dx \right)^s + c\fint_{B_{2r}} |A^{(-1)}(x,F)|^s\,dx+c
\]
for some $c=c(n,p,q,L,s)>0$,
whenever $r\le R_0$ and $B_{2r}\subset\Omega'$.
\end{theorem}

Finally, we have the $C^1$-result under Dini mean continuity assumption, which is the most delicate 
part of the paper and its central contribution. 

\begin{theorem}\label{thm:C1}
Let $A:\Omega \times\R^n\to \R^n$ satisfy the quasi-isotropic $(p,q)$-growth condition, and 
let $u\in W^{1,1}_\loc(\Omega)$ with $|A^{(-1)}(\cdot, Du)|\in L^1_\loc (\Omega)$ be a weak solution to \eqref{maineq}. If $A^{(-1)}$ satisfies \DMA{\gamma} for some $\gamma>2$, then $Du$ is continuous in $\Omega.$
\end{theorem}

\begin{remark}\label{rem:gamma}
The parameter $\gamma$ in the previous theorem can be thought of as the power of the 
mean, as in $\omega_\gamma(r)= \sup_{B_r\subset\Omega} (\fint_{B_r}\fint_{B_r} |M(x)-M(y)|^\gamma\, dx\, dy)^{1/\gamma}$. In the linear case, Dong and Kim \cite{DongKim17} were able to consider the case 
$\gamma=1$, but we require the slightly stronger assumption $\gamma>2$. Thus we included 
an asterisk in Table~\ref{table:summary} to indicate that our general result does not 
quite cover this one paper. In the linear case one has access to additional tools such as a 
weak $L^1$-inequality and a priori estimates whereas we use Calder\'on--Zygmund theory. 
The latter seems to invariably lead to $L^2$-averages and so we believe that our result is 
close to optimal in the non-linear case. See also Proposition~\ref{prop:meanreverse}.
\end{remark}

Let us comment on the Dini mean oscillation version of \aone{}, \DMA{}, and 
the vanishing mean oscillation version of \aone{}, \VMA{}. According to 
Definition~\ref{def:meanoscillation} for $G=A^{(-1)}$ we consider mean averages of 
the quantity
\[
\ttheta(x,B_r) = 
\sup_{\xi\in D(B_r)} \frac{|A(x,\xi) - {A}_{B_r}(\xi)|}
{ |A|_{B_r}(\xi)+\omega(r)}.
\]
The denominator depends on the choice of the ball $B_r$, which make it challenging to find 
the relation between fundamental point-wise \aone{}-type conditions in Definition~\ref{def:continuity} and the mean oscillation type condition and proving properties of the function $\theta_\gamma$ defined in Definition~\ref{def:meanoscillation}. We investigate the mean oscillation type conditions in Section~\ref{sect:DMO}.

The proof of Theorem~\ref{thm:C1} consists of two major steps. The first step is to obtain a comparison estimate in $L^1$-space between the gradients of the weak solutions to the original equation \eqref{maineq} and its approximating autonomous equation in \eqref{eqv} (see Section~\ref{sect:comparison}). Note that the weak solution to \eqref{eqv} has good $C^{1,\alpha}$-regularity estimates by Lemma~\ref{lem:v_regularity}. We emphasize that the comparison estimate in Lemma~\ref{lem:comparison} is sharper than the one we obtained in \cite{HasO22b}. Specially, the exponent of $\Theta$ in Lemma~\ref{lem:comparison} is $2$, whereas the approach used in the proof of \cite[Lemma 6.2]{HasO22b} yields a smaller exponent. 
This sharp estimate allows us to prove the $C^1$-regularity under the \DMA{} condition. 
We also note that the approximation deriving the equation \eqref{eqv} is similar to the one in \cite{HasO22b}, but simpler since we apply the splicing technique only for large values of $|\xi|=t$, rather than for both large and small values. 
The second step is an iteration. We improve upon the iteration argument in \cite{KuuMin14}, leading to Lemma~\ref{lem:induction}, which can be applied for both the Lipschitz regularity and $C^1$-regularity.

The proof of the Calder\'on--Zygmund estimate in Theorem~\ref{thm:CZ} also involves two steps (see Section~\ref{sect:proofCZ}). The first step is, once again, a comparison. Since we are dealing with the non-homogeneous equation \eqref{maineq1}, a new comparison estimate for the gradients of the weak solutions to \eqref{maineq1} and \eqref{eqv} is required, which is provided in Lemma~\ref{lem:approximation}. The next step is to prove the $W^{1,s}$-estimate by estimating integrals of $\phi(x,|Du|)$ over super-level sets. We follow the so-called maximal function free approach, introduced in \cite{AceMin07}. Note that in this process, we must essentially use the \aone{} condition of $A^{(-1)}$.


\section{Preliminaries and notation}\label{sect:preliminaries}

Let $\Omega$ be a bounded domain in $\Rn$ with $n \ge 2$. 
We denote by $B_r(x_0)$ the open ball with center $x_0\in \Rn$ and radius $r>0$. If the center is either clear or irrelevant, we simplify notation to $B_r=B_r(x_0)$. For a set $E\subset \Rn$, $\chi_E$ is the usual \textit{characteristic function} of $E$ such that $\chi_E(x)=1$ if $x\in E$ and $\chi_E(x)=0$ if $x\not\in E$. We denote the H\"older conjugate exponent of $p \in [1,\infty]$ by $p'=\frac{p}{p-1}$. A generic constant denoted by $c>0$ without subscript may vary between appearances.

Let $f, g : E \to \R$ be measurable in $E\subset\R^n$. 
We denote the average of $f$ over $E$ with $0<|E|<\infty$ by 
$(f)_E:=\fint_E f\, dx := \frac{1}{|E|}\int_E f \,dx$. 
The notation $f\lesssim g$ means that
there exists a constant $C>0$ such that $f(y)\le Cg(y)$ all $y\in E$ and 
$f\approx g$ means that $f\lesssim g\lesssim f$.
When $E\subset \R$, we say that $f$ is \textit{almost increasing} on $E$ with constant 
$L\ge 1$ if $f(s)\le L f(t)$ whenever $s,t \in E$ with $s\le t$. If we can choose $L=1$, 
we say that $f$ is \textit{increasing} on $E$. \textit{Almost decreasing} and \textit{decreasing} are defined similarly. 
\textit{Modulus of continuity} refers to a concave and increasing function $\omega:[0,\infty)\to [0,\infty)$ with $\omega(0)=\lim_{r\to 0^+}\omega(r)=0$.

We introduce fundamental conditions on the energy function $\phi:\Omega\times[0,\infty]\to [0,\infty)$. 
We refer to \cite[Chapter~2]{HarH19} for the following definitions and properties.
We start with regularity with respect to the second variable, 
which are supposed to hold for all $x\in \Omega$ and 
a constant $L\ge 1$ independent of $x$. 

\begin{definition} \label{def:ainc}
Let $\phi:\Omega\times[0,\infty]\to [0,\infty)$ and $\gamma\in\R$.
We say that $\phi$ satisfies
\vspace{0.2cm}
\begin{itemize}
\item[\normalfont(aInc)$_\gamma$]\label{ainc} if
$t\mapsto t^{-\gamma}\phi(x,t)$ is almost increasing on $(0,\infty)$ with constant $L$;
\vspace{0.2cm}
\item[\normalfont(Inc)$_\gamma$]\label{inc} if
$t\mapsto t^{-\gamma}\phi(x,t)$ is increasing on $(0,\infty)$;
\vspace{0.2cm}
\item[\normalfont(aDec)$_\gamma$]\label{adec} if 
$t\mapsto t^{-\gamma}\phi(x,t)$ is almost decreasing on $(0,\infty)$ with constant $L$;
\vspace{0.2cm}
\item[\normalfont(Dec)$_\gamma$]\label{dec} if
$t\mapsto t^{-\gamma}\phi(x,t)$ is decreasing on $(0,\infty)$;
\vspace{0.2cm}
\item[\normalfont(A0)] \label{azero} $L^{-1}\leq \phi(x,1)\leq L$.
\end{itemize}
We say that $\phi$ satisfies \ainc{} or \adec{} if it satisfies \ainc{\gamma} or \adec{\gamma}, respectively, for some $\gamma > 1$.

Furthermore, for a vector-valued function $G:\Omega\times \R^M\to \R^N$, we say that $G$ satisfies \ainc{\gamma} or \adec{\gamma} if $\phi(x,t):=|G(x,te)|$ satisfies \ainc{\gamma} or \adec{\gamma}, respectively, with the constant $L$ uniformly in $e\in \R^M$ with $|e|=1$. 
\end{definition}

For $p,q>0$, the conditions \ainc{p} or \adec{q} on $\phi$ with constant $L\geq 1$ are equivalent to 
the following inequalities 
\[
\phi(x,\lambda t)\le L\lambda^p\phi(x,t)
\quad\text{or}\quad
 \phi(x,\Lambda t) \le L \Lambda^q \phi(x,t), \quad \text{respectively},
\]
for all $(x,t)\in \Omega\times [0,\infty)$ and $0\le \lambda\le 1 \le \Lambda$. 
Additionally, \ainc{} or \adec{} 
are equivalent to $\nabla_2$-condition or $\Delta_2$-condition, respectively. 
Although the definition of \azero{} presented above slightly differs from that in \cite{HarH19}, the two definitions coincide when $\phi$ satisfies \adec{}. In the case $\phi(x,\cdot)\in C^1((0,\infty))$, the conditions \inc{p} and \dec{q} for $0<p \le q$ are equivalent to
\[
p \le \frac{t\phi'(x,t)}{\phi(x,t)} \le q \quad \text{for all }\ t\in(0,\infty).
\]

Let us consider increasing functions $\phi,\psi : [0,\infty)\to [0,\infty)$ such that $\phi$ satisfies \ainc{1} and \adec{q} for some $q\ge 1$, and $\psi$ 
satisfies \adec{1}. Then there exist a convex function $\tphi$ and a concave function $\tpsi$ such that $\phi\approx\tilde \phi$ and $\psi\approx \tpsi$ from \cite[Lemma~2.2.1]{HarH19}. 
In turn, Jensen's inequality for $\tphi$ and $\tpsi$ yields that 
\begin{equation*}
\phi\bigg(\fint_\Omega |f|\,dx\bigg) \lesssim \fint_\Omega \phi(|f|)\,dx 
\quad\text{and}\quad 
\fint_\Omega \psi(|f|)\,dx \lesssim \psi\bigg(\fint_\Omega |f|\,dx\bigg)
\end{equation*}
for every $f\in L^1(\Omega)$. Here, the
implicit constants depend on $L$ from \ainc{1} and \adec{q} or \adec{1}, based on the constants arising from the equivalence relation.

\medskip

We define classes of $\Phi$-functions and generalized Orlicz spaces, following \cite{HarH19}. 
Our primary focus is on convex functions relevant to minimization problems and associated PDEs; however, 
the class $\Phiw(\Omega)$ is quite useful for approximating functionals.

\begin{definition}\label{defPhi}
Let $\phi:\Omega\times[0,\infty]\to [0,\infty)$. 
Assume $x\mapsto \phi(x,|f(x)|)$ is measurable for every measurable function $f$ on $\Omega$, 
$t\mapsto \phi(x,t)$ is increasing for every $x\in\Omega$, 
and $\phi(x,0)=\lim_{t\to0^+}\phi(x,t)=0$ and $\lim_{t\to\infty}\phi(x,t)=\infty$
for every $x\in\Omega$. Then $\phi$ is said to be 
\begin{itemize}
\item[(1)] a \textit{$\Phi$-function}, denoted $\phi\in\Phiw(\Omega)$, if it satisfies \ainc{1};
\item[(2)] a \textit{convex $\Phi$-function}, denoted $\phi\in\Phic(\Omega)$, if $t\mapsto \phi(x,t)$ is left-con\-tin\-u\-ous and convex for every $x\in\Omega$.
\end{itemize}
The subsets of $\Phiw(\Omega)$ and $\Phic(\Omega)$ consisting of functions without 
dependence on the first variable (i.e., $\phi(x,t)=\phi(t)$) are denoted by 
$\Phiw$ and $\Phic$, respectively.
\end{definition}

Since convexity implies \inc{1}, we see that $\Phic(\Omega)\subset \Phiw(\Omega)$.
Let us now consider $\phi\in \Phiw(\Omega)$. We define
the (left-continuous) inverse function of $\phi$ with respect to $t$ by
\[
\phi^{-1}(x,t):= \inf\{\tau\geq 0: \phi(x,\tau)\geq t\}.
\]
If $\phi$ is strictly increasing and continuous in $t$, then $\phi^{-1}$ is the usual inverse function.
We also define the conjugate function of $\phi$ by
\[
\phi^*(x,t) :=\sup_{s\geq 0} \, (st-\phi(x,s)).
\] 
From this definition, it follows that \textit{Young's inequality}
\[
ts\leq \phi(x,t)+\phi^*(x, s)
\]
holds for all $ s,t\ge 0$. If $\phi$ satisfies \ainc{p} or \adec{q} for some $p, q>1$, 
then $\phi^*$ satisfies \adec{p'} or \ainc{q'}, respectively \cite[Proposition~2.4.9]{HarH19}.
For simplicity, we write
\[
\phi^+_{B_r}(t):=\sup_{x\in B_r\cap \Omega}\phi(x,t)
\quad \text{and}\quad
\phi^-_{B_r}(t):=\inf_{x\in B_r\cap \Omega}\phi(x,t).
\]

If $\phi\in\Phic(\Omega)$, then $(\phi^*)^*=\phi$ \cite[Theorem~2.2.6]{DieHHR11} and that there exists an increasing and right-continuous function
$\phi':\Omega\times[0,\infty)\to[0,\infty)$, the so-called (right-)derivative of $\phi$, such that 
\[
\phi(x,t)=\int_0^t \phi'(x,s)\, ds.
\]
We recall some results related to this $\phi'$. 

\begin{proposition}[Proposition~3.6, \cite{HasO22}]\label{prop0} 
Let $\gamma>0$ and $\phi\in\Phic(\Omega)$.
\begin{itemize}
\item[(1)] 
If $\phi'$ satisfies \inc{\gamma}, \dec{\gamma}, \ainc{\gamma} or \adec{\gamma}, then $\phi$ satisfies \inc{\gamma+1}, \dec{\gamma+1}, \ainc{\gamma+1} or \adec{\gamma+1}, respectively, with the same constant $L\geq 1$. 
\item[(2)] 
If $\phi$ satisfies \adec{\gamma}, then $(2^{\gamma+1}L)^{-1}t \phi'(x,t) \le \phi(x,t)\le t \phi'(x,t)$.
\item[(3)] 
If $\phi'$ satisfies \azero{} and \adec{\gamma} with constant $L\geq 1$, then $\phi$ also satisfies \azero{}, with constant depending on $L$ and $\gamma$.
\item[(4)] 
$\phi^*(x,\phi'(x,t))\le t\phi'(x,t)$. 
\end{itemize}
\end{proposition}

%



For $\phi\in\Phiw(\Omega)$, the \textit{generalized Orlicz space} (also known as the \textit{Musielak--Orlicz space}) is defined by
\[
L^{\phi}(\Omega):=\big\{f\in L^1_\loc(\Omega):\|f\|_{L^\phi(\Omega)}<\infty\big\},
\] 
with the (Luxemburg) norm 
\[
\|f\|_{L^\phi(\Omega)}:=\inf\bigg\{\lambda >0: \varrho_{L^\phi(\Omega)}\Big(\frac{f}{\lambda}\Big)\leq 1\bigg\},
\quad\text{where}\quad\varrho_{L^\phi(\Omega)}(f):=\int_\Omega\phi(x,|f|)\,dx .
\] 
We denote by $W^{1,\phi}(\Omega)$ the set of functions $f\in W^{1,1}_\loc(\Omega)$ 
with $\|f\|_{W^{1,\phi}(\Omega)}:=\|f\|_{L^\phi(\Omega)}+\big\||Df|\big\|_{L^\phi(\Omega)}<\infty$. 
If $\phi$ satisfies \adec{}, then we note that $f\in L^\phi(\Omega)$ if and only if 
$\varrho_{L^\phi(\Omega)}(f)<\infty$. 
The spaces $L^\phi(\Omega)$ and $W^{1,\phi}(\Omega)$ are reflexive Banach spaces when $\phi$ satisfies \azero{}, \ainc{} and \adec{}.
We denote by $W^{1,\phi}_0(\Omega)$ the closure of 
$C^\infty_0(\Omega)$ in $W^{1,\phi}(\Omega)$. For more information on generalized Orlicz 
and Orlicz--Sobolev spaces, we refer to the monographs 
\cite{CheGSW21, HarH19} and also \cite[Chapter~2]{DieHHR11}.

In recent years, we have studied regularity theory for the general equations \eqref{maineq} with quasi-isotropic $(p,q)$-growth condition \cite{HHL21,HasO22,HasO22b,HasO23,KarLee22}.
In these papers, regularity conditions for the growth function $\phi\in \Phiw(\Omega)$ or a relevant function with respect to the space variable $x$ is given in terms of point-wise oscillation. 
This is in contrast to $(p,q)$-growth approach, where usually the one assumes that 
$\frac qp$ is small, e.g.\ \cite{BelS20}, although see also \cite{CupM_pp}.
Let us recall these assumptions and the regularity results in \cite{HasO23}. 
We have made a slight alteration in that $\omega(r)|\xi|$ previously 
lacked the $|\xi|$; this does not affect which $G$ satisfy the condition, 
but it does impact the functions $\omega$ and $\tomega$, see Example~\ref{eg:Orlicz}.

\begin{definition}\label{def:continuity}
Let $G:\Omega\times \R^M \to \R^N$ with $M,N\in\N$, 
$r\in (0,1]$, and $\omega, \tomega:[0,1]\to [0,L]$ with $L>0 $. Consider the claim 
\[
|G(x,\xi)-G(y,\xi)|\leq \tomega(r)\big(|G(y,\xi)|+ \omega(r)|\xi|\big) 
\quad\text{when }\ |G(y,\xi)|\in [0,|B_r|^{-1}]
\]
for all $x,y\in B_r$, $B_{2r}\subset \Omega$ and $\xi\in\R^M$.
We say that $G$ satisfies
\begin{itemize}
\item[\normalfont(A1)]\label{aone}
if 
the claim holds with 
$\omega=\tomega\equiv L$;
\item[\normalfont(SA1)]\label{SA}
if there exists 
a modulus of continuity 
$\omega$ such that the claim holds with 
$\tomega\equiv L$;
\item[\normalfont(VA1)]\label{VA}
if there exists 
a modulus of continuity 
$\omega=\tomega$ such that the claim holds;
\item[\normalfont(wVA1)]\label{wVA}
if, for every $\epsilon > 0$, $G$ satisfies \VA{} with the range condition replaced by 
$|G(y,\xi)|^{1+\epsilon}\in [0,|B_r|^{-1}]$, 
with moduli of continuity $\omega_\epsilon:=\omega=\tomega$ depending on $\epsilon$, 
but with a common $L$ independent of $\epsilon$. 
\end{itemize}
We also use the definition for $\phi:\Omega\times [0,\infty)\to[0,\infty)$ by interpreting $\phi(x,\xi)=\phi(x,|\xi|)$.
\end{definition}

We refer to Section~8 of \cite{HasO22}
for examples of functions satisfying these conditions. 
With these point-wise assumptions, we proved maximal regularity results in 
\cite{HasO22, HasO22b, HasO23}.

\begin{theorem}
[Theorems~1.2 and 4.1, \cite{HasO22b}]\label{thm:C1alpha}
Let $A$ satisfy the quasi-isotropic (p,q)-growth condition and let $u\in W^{1,1}_{\loc}(\Omega)$ with $|A^{(-1)}(\cdot, Du)|\in L^1_{\loc}(\Omega)$ be a local weak solution to \eqref{maineq}. 
\begin{itemize}
\item[$\mathrm{(1)}$] If $A^{(-1)}$ satisfies \aone{}, then $u\in C^{0,\alpha}_{\loc}(\Omega)$ for some $\alpha\in (0,1)$ depending on $n,p,q$ and $L$. 
\item[$\mathrm{(2)}$] If $A^{(-1)}$ satisfies \wVA{}, then $u\in C^{0,\alpha}_{\loc}(\Omega)$ for every $\alpha\in (0,1)$.
\item[$\mathrm{(3)}$] If $A^{(-1)}$ satisfies \wVA{} with Hölder-continuous $\omega_\epsilon$ for every $\epsilon>0$, then $u\in C^{1,\alpha}_{\loc}(\Omega)$ for some $\alpha\in (0,1)$ depending on $n,p,q$ and $L$.
\end{itemize}
\end{theorem}

\section{Mean oscillation conditions}\label{sect:DMO}

We introduce the mean oscillation variants of the \aone{}-condition, specifically, the 
\textit{vanishing mean-A1 condition} \VMA{} and the \textit{Dini mean-A1 condition} 
\DMA{}.  
For a vector valued function $G:\Omega\times \R^M \to \R^N$ and $U\subset \Omega$, denote
\[
G^+_U(\xi):=\sup_{x\in U} |G(x,\xi)|,
\quad G^-_U(\xi):=\inf_{x\in U} |G(x,\xi)|,
\]
\[
G_{B_r}(\xi) := \fint_{U} G(x,\xi)\,dx
\quad \text{and}\quad
|G|_{B_r}(\xi) := \fint_{U} |G(x,\xi)|\,dx.
\]

\begin{definition}\label{def:meanoscillation}
Let $G:\Omega\times \R^M \to \R^N$ with $M,N\in\N$, $r\in (0,1]$, 
$\omega:[0,1]\to [0,1]$, and 
\[
\ttheta(x,B_r) := 
\sup_{\xi\in D(B_r)} \frac{|G(x,\xi) - {G}_{B_r}(\xi)|}
{ 
|G|_{B_r}(\xi) 
+\omega(r) |\xi|} 
\quad \text{for } \ x\in B_r,
\]
where $D(B_r):= \{\xi\in\R^M:  |G|_{B_r}(\xi) \le |B_r|^{-1}\}$.
For $\gamma\ge1$, we consider
\[
\theta_\gamma(r) := \left(\sup_{B_{2r}\subset \Omega} \fint_{B_r} \ttheta(x,B_r)^\gamma \,dx\right)^\frac{1}{\gamma}.
\]
We say that $G$ satisfies:
\begin{itemize}
\item[\normalfont(VMA1)]\label{VMA}
if there exists a nondecreasing $\omega$ such that 
$\displaystyle
\lim_{r\to 0^+} [ \theta_1(r)+\omega(r) ]=0$.
\item[\normalfont(DMA1)$_\gamma$]\label{DMA}
if there exists a nondecreasing $\omega$ such that 
$\displaystyle
\int_0^1 [\theta_\gamma(r)+\omega(r) ]\frac{dr}{r} <\infty$.
\end{itemize}
\end{definition}

We refer to Section~\ref{sect:examples} for several examples of energies satisfying these 
conditions. There we show that all the cases from Table~\ref{table:summary} are included as special 
cases. By Hölder's inequality, \DMA{\gamma} is monotone in $\gamma$. Furthermore, if the 
case $\gamma=\infty$ is understood as a supremum, then $\theta_\infty$ is closely related 
to the quantity in the point-wise conditions of Definition~\ref{def:continuity}.

\begin{remark}
In Definition~\ref{def:meanoscillation}, we consider the $\gamma$-mean 
of $\ttheta$ for some $\gamma\ge 1$. Proposition~\ref{prop:meanreverse} implies that if we assume that 
$\theta$ is nondecreasing, then it suffices to consider the case $\gamma=1$. 
However, in Proposition~\ref{prop:Lorentz} we see that the current weaker assumption 
is more useful to cover all special cases. 
\end{remark}

\begin{remark}\label{rmk:wDMA}\label{wDMA}\label{wVMA}
Similarly to the vanishing condition \VA{}, we can also consider weak versions (wDMA1) and (wVMA1) 
of \DMA{} and \VMA{}. However, we are not aware of any examples 
where these conditions would be needed. Thus we will for the sake of simplicity not 
consider them in what follows. Furthermore, we could define a point-wise condition (DA1) of 
Dini-type, but this is already obsolete since we directly handle the more general condition 
\DMA{}.
\end{remark}

We show that the \DMA{1} condition also implies the continuity of $G(x,\xi)$ in the $x$ variable and the \SA{} condition. For this we need Spanne's result relating modulus of mean continuity 
with modulus of continuity:

\begin{lemma}[Corollary~1 with its remark, \cite{Spa65}]\label{lem:Spanne}
Assume that $f\in L^1(\Omega)$ satisfies the Dini mean continuity condition
\[
\int_0^1 \bar\omega_f(r)\frac{dr}{r} <\infty
\qquad \text{where}\qquad
\bar\omega_f(r) := \sup_{B_r\subset \Omega} \fint_{B_r} |f(x)-(f)_{B_r}| \,dx.
\]
Then $f$ is continuous with modulus $\omega_f(r)=c\int_0^r \bar\omega_f(s)\,\frac{ds}{s}$ for some $c>0$.
\end{lemma}

From this lemma we can get an intuition of how the point-wise and mean moduli are related. 
First, if $\bar\omega_f(r)\lesssim r^\beta$, then also $\omega_f(r)\lesssim r^\beta$, so there is no 
difference between the cases here. However, if $\bar\omega_f(r)=(\log\frac1r)^{-\alpha}$, $\alpha>1$, then 
$\omega_f(r) \approx (\log\frac1r)^{1-\alpha}$, so there is a loss of a logarithm between 
the cases. In particular, if $\alpha\in (1, 2]$, then the point-wise modulus $\omega_f$ does 
not satisfy the Dini condition even though the mean oscillation modulus $\bar\omega_f$ does.

\begin{proposition}\label{prop:DMA} 
Let $G:\Omega\times \R^M\to \R^N$ with $M,N\in \mathbb{N}$ satisfy \azero{}, \ainc{p} and \adec{q} for some $1\le p\le q$, 
as well as \DMA{1}.
\begin{enumerate}
\item
For each $\xi\in \R^M$, $G(\cdot,\xi)$ is continuous.
\item 
$G$ satisfies \SA{} with the same modulus of continuity $\omega$ and the relevant constant $L\ge 1$ depending on $n,p,q,L,$ and $\theta_1$. 
\end{enumerate}
\end{proposition}
\begin{proof}
Note that 
$G^+_U(\xi)=\sup_{x\in U}|G(x,\xi)|$ for any $U\subset \Omega$
is finite by \azero{} and \adec{q}.

(1) Fix $\xi\in\R^M\setminus \{0\}$. Then there exists $r_0>0$ such that $G^+_\Omega(\xi) \le |B_{r_0}|^{-1}$. For $r\in (0,r_0]$, from the definition of $\tilde \theta$ in Definition~\ref{def:meanoscillation}, $ |G(x,\xi)-G_{B_r}(\xi)| \le (G^+_\Omega(\xi)+|\xi|) \tilde\theta(x,B_r)$ for every $x\in B_r $ and $B_{2r}\subset\Omega$. Hence
\[\begin{split}
\sup_{B_{2r}\subset\Omega}\fint_{B_r} |G(x,\xi)-G_{B_r}(\xi)|\,dx 
& \le (G^+_\Omega(\xi)+|\xi|) \sup_{B_{2r}\subset\Omega}\fint_{B_r} \tilde\theta(x,B_r)\, dx \\
& = (G^+_\Omega(\xi)+|\xi|) \theta_1(r).
\end{split}
\]
Therefore, by Lemma~\ref{lem:Spanne}, the \DMA{1} condition yields the continuity of $G(\cdot,\xi)$.

(2) Fix $B_{2r}\subset \Omega$ with small $r>0$ to be determined later. 
Consider 
$\xi\in \R^n$ 
with $G^+_{B_r}(\xi)\le |B_{r}|^{-1}$. 
Then for any ball $B_{\rho} \subset B_r$, since 
$
|G|_{B_{\rho}}(\xi)
\le G^+_{B_r}(\xi) \le |B_r|^{-1} \le |B_{\rho}|^{-1}$, it follows from Definition~\ref{def:meanoscillation} that
\[
|G(x,\xi)-G_{B_\rho}(\xi)| 
\le (
|G|_{B_{\rho}}(\xi)
+ \omega(\rho)|\xi|) \tilde\theta(x,B_{\rho}) 
\lesssim (G^+_{B_r}(\xi)+ \omega(r)|\xi|) \tilde\theta(x,B_{\rho})
\]
for any $x\in B_\rho$. Hence by Definition~\ref{def:meanoscillation},
\[
\sup_{B_\rho\subset B_r}\fint_{B_\rho}|G(x,\xi)-G_{B_\rho}(\xi)| \, dx
\lesssim \left(G^+_{B_r}(\xi)+ \omega(r)|\xi|\right) \, \theta_1(\rho).
\]
Therefore, in view of Lemma~\ref{lem:Spanne}, the \DMA{1} condition yields that
\[
|G(x,\xi)-G(z,\xi)| \le c_1 \left(G^+_{B_r}(\xi)
+ \omega(r)|\xi|\right) \int_0^r\theta_1 (\rho)\frac{d\rho}{\rho}, 
\quad \text{for any }\ x, z\in B_r
\]
for some constant $c_1\ge 0$ depending on $p$, $q$, $L$ and $\theta_1$. 
Choose $r\in (0, r_1]$, where $r_1$ is determined by $\int_0^{r_1}\theta (\rho)\frac{d\rho}{\rho} = \frac{1}{2c_1}$. Then 
\[
G^+_{B_r}(\xi) - G^-_{B_r}(\xi) 
\le \sup_{x, z\in B_r} |G(x,\xi)-G(z,\xi)|
\le \tfrac12 (G^+_{B_r}(\xi) + \omega(r)|\xi|),
\]
which implies that
\[
G^+_{B_r}(\xi) \le 2 G^-_{B_r}(\xi) +\omega(r)\,|\xi|.
\]
This implies the desired \SA{}-inequality, when $G^+_{B_r}(\xi)\le |B_{r}|^{-1}$. 

Suppose then that $G^-_{B_r}(\xi)\le |B_{r}|^{-1}$ and let $s\in (0,1]$ be the largest 
number with $G^+_{B_r}(s\xi)\le |B_{r}|^{-1}$. The case $s=1$ was handled above. If 
$s<1$, then $G^+_{B_r}(s\xi)= |B_{r}|^{-1}$ and it follows from the earlier case that 
\[
|B_{r}|^{-1}=G^+_{B_r}(s\xi) \le 2 G^-_{B_r}(s\xi) +\omega(r)|s\xi|.
\]
Since $r\le 1$, it follows from \azero{} that 
$|s\xi|\gtrsim 1$.
Then \azero{} and $\omega(r)\le 1$ 
yield $\omega(r)|s\xi|\lesssim G^-_{B_r}(s\xi)$. Using also \adec{q}, we find that 
\[
|B_r|^{-1} \le 2 G^-_{B_r}(s\xi) +\omega(r)|s\xi|\le cG^-_{B_r}(s\xi)\le c_2 s^p G^-_{B_r}(\xi)
\le c_2s^p |B_r|^{-1}.
\]
Hence $s\ge c_2^{-1/p}$.  Returning to the earlier case, we then conclude that 
\[
G^+_{B_r}(\xi)
\le
Ls^{-q}G^+_{B_r}(s\xi) 
\le 
Ls^{-q} (2 G^-_{B_r}(s\xi) +\omega(r)|s\xi|)
\le 
2L c_2^{q/p} (G^-_{B_r}(\xi) +\omega(r)|\xi|),
\]
which implies that \SA{}-inequality with the correct assumption $G^-_{B_r}(\xi)\le |B_{r}|^{-1}$
for $r\in (0, r_1]$. When $r\in (r_1,1]$, we obtain the conclusion using a chain of
a fixed number $\lceil \frac1{r_1}\rceil$ of balls of radius $r_1$ and the above argument. 
\end{proof}

The function $\theta$ in the above definition is not assumed to be increasing. This makes it 
harder to estimate geometric series based on $\theta$, but the next lemma shows that it is possible. 

\begin{lemma}\label{lem:theta-sum}
Let $G:\Omega\times \R^M \to \R^N$ satisfy \azero{}, \ainc{p} and \adec{q} for some $1\le p\le q$ and \DMA{\gamma} for some $\gamma\ge 1$.
There exists $\delta_1=\delta_1(n,p,q,L,\theta_\gamma)\in(0,1)$ such that
for every $\delta\in (0,\delta_1)$ and $r\in (0,1)$ with $|B_{2r}|\le 1$,
\[
\sum_{k=0}^\infty \theta_\gamma(\delta^kr) 
\le
\frac {c}{
\delta^{\frac{4n+1}{\gamma}}
\log(\delta^{-1})}\int_0^r \theta_\gamma(\rho) \frac{d\rho}{\rho}
\]
for some 
$c=c(n,\gamma)>0$.
\end{lemma}
\begin{proof}
Let $\tilde r:=\delta^kr$ with $k\ge 1$ and 
$\rho\in (\delta \tilde r,\sqrt{\delta} \tilde r)$.
Fix $B_{2r}\subset\Omega$. Since $G$ 
satisfies
\SA{} by Proposition~\ref{prop:DMA}(2), 
\[
G^+_{B_{3\tilde r/2}}(\xi) \le \tilde L \big( G^-_{B_{3\tilde r/2}}(\xi)+ \omega(\tilde r) | \xi|\big)
\] 
for any 
$\xi\in\R^M$ with $G^-_{B_{3\tilde r/2}}(\xi) \le |B_{\tilde r}|^{-1}$, and for some $\tilde L\ge 1$ depending on $n$, $p$, $q$ and $\theta_1$. 
Fix $\xi\in \R^M$ such that 
$
| G|_{B_{\tilde r}}(\xi)\le |B_{\tilde r}|^{-1}$. 
Then the preceding inequality yields for any $B_\rho(z)\subset B_{3 \tilde r/2}$ 
with $z\in \overline{B_{\tilde r}}$ that
\begin{equation}\label{Lem:theta-sum_pf1}
|G|_{B_\rho(z)}(\xi) \le \tilde L \big(|G|_{B_{\tilde r}}(\xi)+\omega(\tilde r)|\xi|\big)
\end{equation}
Therefore, since $\omega(r)\le 1$,
\[
|G|_{B_\rho(z)}(\xi)
\le \tilde{L} \left( |G_{B_{\tilde r}}(\xi)| + |\xi|\right) \le \tilde L(L +1) |B_{\tilde r}|^{-1} 
\le |B_\rho|^{-1} 
\]
provided that 
$\rho \le  \min\{[\tilde L(L + 1)]^{-1/n},\frac12\}r$. 
Thus $D(B_{\tilde r})\subset D(B_\rho(z))$ for $z\in \overline{B_{\tilde r}}$.

Denote $H_{\xi}(x):=\frac{G(x,\xi)}{|G|_{B_{\tilde r}}(\xi)+\omega(r) |\xi|}$. 
We connect the integrals of $H_\xi$ in the two balls by a chain argument.
We cover $B_{\tilde r}$ with $c(n) \delta^{-n}$ balls $B^i$ of radius $\frac{\rho}{2}$, whose centers are in $\overline{B_{\tilde r}}$, with the property that 
every $x,y\in B_{\tilde r}$ with $|x-y|\le \frac \rho4$ belong to some ball in the cover. 
For $x\in B^i$ and $y\in B^j$ with $i\neq j$ and $|x-y|> \frac{\rho}{4}$ we use the estimate
\[\begin{split}
|H_\xi(x)-H_\xi(y)| & \le
|H_\xi(x)- (H_\xi)_{B_{\rho/2}(x_0)}|
+\sum_{l=0}^m|(H_\xi)_{B_{\rho/2}(x_{l-1})}-(H_\xi)_{B_{\rho/2}(x_l)}| \\
&\qquad + |(H_\xi)_{B_{\rho/2}(x_{m})}-H(y)|
\end{split}\]
where $x_0=x$, $x_m=y$ and $x_l\in [x, y]:=\{tx+(1-t)y:t\in[0,1]\}$ for $l=1,\dots,m-1$ satisfy $\frac{\rho}{4}< |x_l-x_{l-1}|\le \frac{\rho}{2}$ for $l=1,2,\dots,m$. Note that 
$m\le 8 \delta^{-1}$. Using $D(B_{\tilde r})\subset D(B_\rho(x_l))$ and Hölder's 
inequality in the second step, we have
\[\begin{split}
&\fint_{B^i}\fint_{B^j} \sup_{\xi\in D(B_{\tilde r})} |H_\xi(x)-H_\xi(y)|^\gamma\, dy\, dx \\
&\le \fint_{B^i} \sup_{\xi\in D(B_{\tilde r})} |H_\xi(x)- (H_\xi)_{B_{\rho/2}(x_0)}|^\gamma\, dx 
+ \fint_{B^j} \sup_{\xi\in D(B_{\tilde r})} |(H_\xi)_{B_{\rho/2}(x_{m})}-H_\xi(y)|^\gamma\, dy \\
&\qquad +\sum_{l=1}^{m} \sup_{\xi\in D(B_{\tilde r})} |(H_\xi)_{B_{\rho/2}(x_{l-1})}-(H_\xi)_{B_{\rho/2}(x_l)}|^\gamma \\
&\lesssim 
\sum_{l=0}^{m} \fint_{B_{\rho}(x_l)} \fint_{B_{\rho}(x_l)}\sup_{\xi \in D(B_{\rho}(x_l))} |H_\xi(x)- H_\xi (y)|^\gamma \, dy \, dx.
\end{split}
\]

We can connect integals over $\ttheta$ and $H_\xi$. We estimate 
\[
\fint_{B_{\tilde r}} \ttheta(x,B_{\tilde r})^\gamma\, dx
\approx
\fint_{B_{\tilde r}}\fint_{B_{\tilde r}} \sup_{\xi \in D(B_{\tilde r})} |H_\xi(x)-H_\xi (y)|^\gamma\, dy\, dx
\]
and, using also \eqref{Lem:theta-sum_pf1}, we see that 
\[
\fint_{B_{\rho}(x_l)}\fint_{B_{\rho}(x_l)} \sup_{\xi \in D(B_{\rho}(x_l))} |H_\xi(x)-H_\xi (y)|^\gamma\, dy\, dx
\lesssim 
\fint_{B_{\rho}(x_l)} \ttheta(x,B_\rho(x_l))^\gamma\, dx 
\le \theta_\gamma(\rho)^\gamma.
\]
Combining these with the estimate from the previous paragraph and $m\lesssim \frac1\delta$, 
we find that 
\[\begin{split}
\fint_{B_{\tilde r}} \ttheta(x,B_{\tilde r})^\gamma\, dx
&\lesssim \delta^{-2n} \sum_{i,j} \fint_{B^i}\fint_{B^j} \sup_{\xi\in D(B_{\tilde r})} |H_\xi(x)-H_\xi(y)|^\gamma\, dy\, dx  \\
&\lesssim \delta^{-2n} \sum_{i,j}  \delta^{-1}\theta_\gamma(\rho)^\gamma \lesssim \delta^{-4n-1} \theta_\gamma(\rho)^\gamma.
\end{split}\]
We have thus shown that $\theta_\gamma(\tilde r) \lesssim 
\theta_\gamma(\rho)$
for $\rho\in (\delta \tilde r, \sqrt{\delta} \tilde r)$ with $\tilde r=\delta^k r$ and $\delta \le \min\{[\tilde L(L + 1)]^{-1/n},1/2\}^2=: \delta_1$. Thus we calculate 
\[
\int_{\delta^{k+1}r}^{\delta^kr} \theta_\gamma(\rho) \frac{d\rho}{\rho}
\gtrsim 
\delta^{\frac{4n+1}{\gamma}}
 \theta_\gamma(\delta^{k}r) \int_{\delta^{k+1}r}^{\sqrt{\delta}\delta^kr} \frac{d\rho}{\rho}
=\frac{\delta^{\frac{4n+1}{\gamma}}}{2} \log(\delta^{-1}) \theta_\gamma(\delta^{k}r).
\] 
Adding the inequalities over $k$ yields the claim. 
\end{proof}

We derive some properties of  the supremal counterpart $\theta^*_\gamma$ of $\theta_\gamma$.

\begin{proposition}\label{prop:meanreverse}
Let $G: \Omega\times \R^M\to \R^N$ satisfy \azero{}, \ainc{p} and \adec{q} for some $1\le p\le q$
as well as \SA{}. With $\theta_\gamma$ from Definition~\ref{def:meanoscillation} for non-decreasing 
$\omega$ and $\gamma\ge 1$ we define 
$\theta^*_\gamma(r):=\sup_{\rho\in (0,r]}\theta_\gamma(\rho)$. Then $\theta^*_\gamma \le L_\gamma \theta^*_1$ for some $L_\gamma>0$ depending on $n$, $p$, $q$ and $\gamma$.
We also note two consequences of this.
\begin{enumerate}
\item 
If $\displaystyle\int_0^1 \theta^*_1(\rho)\frac{d\rho}\rho<\infty$, then \DMA{\gamma} holds for any $\gamma\ge 1$.
\item 
If $G$ satisfies \SA{} and \VMA{}, then $\displaystyle\lim_{r\to 0} \theta^*_\gamma(r)=0$ for any $\gamma\ge 1$. 
\end{enumerate}
\end{proposition}

\begin{proof}
We follow ideas from the proof of the John--Nirenberg inequality. 
In Definition~\ref{def:meanoscillation}, the supremum is taken over balls, so we apply the 
Vitali covering lemma instead of the Calder\'on--Zygmund decomposition. 
Fix $r\in (0,1]$ and abbreviate $\theta^*:=\theta^*_1(r)$.
It follows from the definition of $\theta^*$ that
\begin{equation}\label{lem:meanreverse_pf1}
 \fint_{B_\rho(y)} \ttheta(x, B_\rho(y)) \,dx \le \theta^* \quad \text{for all }\ B_{2\rho}(y) \subset \Omega \ \text{ with }\ \rho\in(0,r].
\end{equation}
Let $\delta\in(0,\frac1{10})$ be a small constant to be determined later, and $\beta_n>1$ be a constant depending only $n$ such that $|B_\rho| \le \beta_n |B_r\cap B_\rho (y)|$ for any $y\in B_r$ and $\rho\in(0, 2r]$.
Denote $E(t, U) := \{x\in U \cap B_r : \ttheta(x, U) > t \}$.
 
We first estimate the measure of the set $E(\beta_n\delta^{-n}\theta^*,B_r)$.
Observe that for every $y\in B_r$ and $\rho\in (\delta r, r]$,
\[
\fint_{B_r\cap B_\rho(y)} \ttheta (x,B_r)\,dx < \beta_n\delta^{-n}\fint_{B_r} \ttheta (x,B_r)\,dx \le \beta_n\delta^{-n}\theta^* ,
\]
and that for almost every $y\in E(\beta_n\delta^{-n}\theta^*,B_r)$
\[
\lim_{\rho\to 0^+} \fint_{B_r\cap B_\rho(y)} \ttheta (x,B_r)\,dx > \beta_n\delta^{-n}\theta^*.
\]
Then by the continuity of the integral with respect to $\rho$, for almost every $y\in E(\beta_n\delta^{-n}\theta^*,B_r)$ one can find $\rho_y\in (0, \delta r]$ such that 
\[
 \fint_{B_r\cap B_{\rho_y}(y)} \ttheta (x,B_r)\,dx = \beta_n\delta^{-n} \theta^*
 \quad \text{and}\quad
 \fint_{B_r\cap B_\rho(y)} \ttheta (x,B_r)\,dx < \beta_n\delta^{-n} \theta^* \quad \forall \rho\in (\rho_y,r].
\]
Therefore, by the Vitali covering lemma, there exist $y^1_j \in E(\beta_n\delta^{-n}\theta^*, B_r)$ 
and $\rho^1_j\in (0,\delta r]$ for $j=1, 2, \ldots$ such that the balls 
$B^1_j:=B_{\rho^1_j}(y^1_j)$ are mutually disjoint,
\[
\bigcup_{j=1}^\infty 5 B^1_j \supset E(\beta_n\delta^{-n}\theta^*, B_r) \setminus \mathcal N^0, 
\]
where $5B^1_j:= B_{5 \rho^1_j}(y^1_j)$ and $\mathcal N^0$ is some measure zero set,
\[
 \fint_{B_r\cap B^1_j} \ttheta (x,B_r)\,dx = \beta_n \delta^{-n} \theta^*
 \quad\text{and}\quad 
 \fint_{B_r\cap 5 B^1_j} \ttheta (x,B_r)\,dx < \beta_n \delta^{-n} \theta^*;
\]
the earlier setup is applicable, since $5\rho^1_j\le 5\delta r \le r$.
The equation in above together with \eqref{lem:meanreverse_pf1} implies
\[
\sum_{j=1}^\infty | B^1_j\cap B_r| = \frac{\delta^n}{\beta_n\theta^*} \sum_{j=1}^\infty \int_{B_r\cap B^1_j} \ttheta (x,B_r)\,dx \le \frac{\delta^n}{\beta_n\theta^*} \int_{B_r} \ttheta (x,B_r)\,dx 
\le  \frac{\delta^n}{\beta_n} |B_r|,
\]
so that 
\[
|E(\beta_n \delta^{-n}\theta^*,B_r)| 
\le \sum_{j=1}^\infty |5 B^1_j\cap B_r| 
\le 5^n \sum_{j=1}^\infty | B^1_j| 
\le \beta_n 5^n \sum_{j=1}^\infty | B^1_j\cap B_r|  \le (5 \delta)^n |B_r|.
\]

We can use the same procedure in the ball $5B^1_j$ to conclude that 
\[
|E(\beta_n \delta^{-n}\theta^*, 5B^1_j)| 
 \le \sum_{l=1}^\infty |5 B^1_{j,l}\cap B^1_{j}|
   \le (5 \delta)^n  |5B^1_j| 
   = (5^2 \delta)^n  |B^1_j|,
\]
where $B^1_{j,l}$, $l\in \mathbb{N}$, are mutually disjoint balls whose five-fold dilates cover $B^1_j$.
This holds in every $B^1_j$, so we estimate 
\[\begin{split}
\sum_{j=1}^\infty|E(\beta_n \delta^{-n}\theta^*, 5B^1_j)| 
\le
\sum_{j=1}^\infty \sum_{l=1}^\infty |5 B^1_{j,l}\cap B^1_j|
&\le 
(5^{2}\delta)^{n} \beta_n \sum_{j=1}^\infty |B^1_j|\\
&\le 
(5\delta)^{n} \beta_n \sum_{j=1}^\infty | 5B^1_j\cap B_r|
\le  
 (5 \delta)^{2n}\beta_n |B_r|.
\end{split}\]
Denote by $(B^2_j)$ the sequence of balls 
whose five-fold dilates cover the first generation balls $B^1_j$. 
Repeating this process, for each $k\in \mathbb N$, we get a sequence of balls $(B^k_j)$ in $B_{2r}$ such that $\bigcup_{j=1}^\infty 5 B^k_j \supset E(\beta_n\delta^{-n}\theta^*, B_r) \setminus \mathcal N^k$ 
for some measure zero set $\mathcal N^k$, 
\begin{equation}\label{lem:meanreverse_pf3}
\fint_{5B^{k-1}_{\tilde j}\cap 5 B^k_j} \ttheta (x,5B^{k-1}_{\tilde j})\,dx 
< \beta_n\delta^{-n}\theta^*,
\end{equation}
where $\tilde j$ is such that $B^{k-1}_{\tilde j}$ is the predecessor of $B^k_j$, and 
\begin{equation}\label{lem:meanreverse_pf4}
\sum_{j=1}^\infty|E(\beta_n \delta^{-n}\theta^*, 5B^k_j)| \le  (5 \delta)^{(k+1)n}\beta_n^{k} |B_r|.
\end{equation}
We also note that the condition $\delta <\frac1{10}$ implies $B^k_j\subset B_{2r}$ for any $k$ and $j$. 
We impose on $\delta$ the condition $(5 \delta)^{n}\beta_n <1$.

By the same argument as in the beginning of the proof of Lemma~\ref{lem:theta-sum}, we see from the \SA{} condition of $G$ that $|G|_{B_r}(\xi) +\omega(r)|\xi| \approx | G|_{B_\rho(y)}(\xi) +\omega(r)|\xi|$ 
for every $B_\rho (y)\subset B_{3r/2}$ and 
$\xi \in D(B_r)$, and that $D(B_r)\subset D(B_\rho(y))$ for any $B_\rho(y)\subset B_{3r/2}$ with $y\in \overline{B_r}$ and $\rho\in(0,\delta r]$ by choosing sufficiently small $\delta$, see the proof of the previous lemma.
Now, fix $k\in\mathbb N$. 
For any $j\in \mathbb N$, let $B^k:=5B^k_{j}$, let
$B^l:=5B^l_{j_l}$ be its predecessor at level $l=1,2,\dots,k-1$, and $B^0=B_r$. 
Denote by $\rho_l$ the radius of $B^l$.
Note that $|B^{l}|\le \beta_n |B^{l}\cap B^{l-1}|$. Then for $x\in B^k$,
\[\begin{aligned}
&|G(x,\xi) - G_{B_r}(\xi)|
\le
|G(x,\xi) - G_{B^k\cap B^{k-1}}(\xi)|
+
|G_{B^k\cap B^{k-1}}(\xi) - G_{B^{k-1}\cap B^{k-2}}(\xi)|\\
&\qquad\qquad\qquad\qquad\qquad +\cdots+
|G_{B^2\cap B^{1}}(\xi) - G_{B^{1}\cap B_r}(\xi)| +
|G_{B^1\cap B_r}(\xi)-G_{B_r}(\xi)|\\
&\qquad \le 
|G(x,\xi) - G_{B^k}(\xi)|
+ \sum_{l=1}^k \Big( | G_{B^l\cap B^{l-1}}(\xi)- G_{B^l}(\xi) |
+ | G_{B^l\cap B^{l-1}}(\xi)- G_{B^{l-1}}(\xi) | \Big).
\end{aligned}\]
Since $\omega$ is non-decreasing, we have 
$|G|_{B_r}(\xi)+\omega(r)|\xi| \gtrsim | G|_{B_\rho(y)}(\xi)+\omega(\rho)|\xi|$
from before. Hence, using the above observations, \eqref{lem:meanreverse_pf1} and \eqref{lem:meanreverse_pf3},
\[\begin{aligned}
\ttheta(x,B_r) 
&= \sup_{\xi \in D(B_r)} \frac{|G(x,\xi) - G_{B_r}(\xi)|}{ | G_{B_r}(\xi)| +\omega(r)|\xi|}\\
&\lesssim
 \ttheta(x, B^k) 
+ \sum_{j=1}^k\left(\beta_n \fint_{B^j}\ttheta (x, B^{j}) \,dx+ \fint_{B^j\cap B^{j-1}}\ttheta (x, B^{j-1}) \,dx\right)\\
&\lesssim
 \ttheta(x, B^k) 
+ k\beta_n(\delta^{-n}+1) \theta^* .
\end{aligned}\]
Let $\alpha_1\ge 1$ denote the implicit constant in the above estimate. Then it follows that $
E(2\delta^{-n}\beta_n \alpha_1(k+1)\theta^*,B_r)\cap B^k 
\subset E(2\beta_n\delta^{-n}\theta^*,B^k)$
for each ball 
$B^k=5B^k_j$, $j\in \mathbb N$.
Therefore, by the covering property and \eqref{lem:meanreverse_pf4}, we obtain 
\[
|E(2\beta_n\delta^{-n}\alpha_1(k+1)\theta^*,B_r)| \le 
\sum_{j=1}^\infty |E(2\beta_n\delta^{-n}\theta^*, 5B^k_j)| 
\le (5\delta)^{(k+1)n}\beta_n^{k} |B_r|.
\]
%
Finally, when $2\beta_n\delta^{-n}\alpha_1(k+1)\theta^* \le t < 
2\beta_n\delta^{-n}\alpha_1(k+2)\theta^*$ for some $k\in \mathbb N$, we find that
\[\begin{split}
\big|\big\{x\in B_r : \ttheta (x, B_r) > t \big\}\big| 
=|E(t, B_r)|
&\le (5\delta)^{n(k+1)} \beta_n^{k}|B_r| \\
&\le (5\delta)^{n} \exp\Big(- \underbrace{\frac{\delta^n}{6\beta_n\alpha_1} \ln\Big(\frac{1}{(5\delta)^n\beta_n}\Big)}_{=:\alpha_2} \frac{t}{\theta^*}  \Big)|B_r|.
\end{split}\]
As in the John--Nirenberg inequality, this implies that 
\[\begin{aligned}
\fint_{B_r} \ttheta(x, B_r)^\gamma \,dx
& = \frac{1}{|B_r|}\int_0^\infty \gamma t^{\gamma-1} |\{x\in B_r: \ttheta(x, B_r) > t\}| \,dt \\
&\le \gamma (5\delta)^n \int_0^\infty t^{\gamma-1}\exp \left(-\frac{\alpha_2}{\theta^*}t\right) \,dt = \gamma (5\delta)^n \Gamma(\gamma)  \left(\frac{\theta^*}{\alpha_2}\right)^\gamma,
\end{aligned}\]
where $\Gamma$ is the gamma function. 
Thus $\theta_\gamma\le L_\gamma \theta^*_1$. Claims (1) and (2) follow directly 
from this inequality.
\end{proof} 


\section{Examples}\label{sect:examples}

In this section we consider what the Dini mean oscillation condition \DMA{\gamma} and 
the vanishing mean oscillation condition \VMA{} means for specific model energies. 
The first example shows how they are related to \VA{}. 

\begin{example}[Point-wise conditions]\label{eg:Orlicz}
Let $G:\Omega\times \R^M\to\R^N$ satisfy \VA{} with $\tilde \omega=\omega=\omega_V$. 
Fix a ball $B_{2r}\subset\Omega$ and $\xi\in \R^M$ with 
$|G|_{B_r}(\xi) \le |B_r|^{-1}$. Choosing $y\in B_{r}$ such that 
$|G|_{B_r}(\xi) = |G(y,\xi)|$, we see from \VA{} that 
\[\begin{split}
|G(x,\xi)-G_{B_r}(\xi)|&\leq |G(x,\xi)-G(y,\xi)|+\fint_{B_r}|G(y,\xi) - G(z,\xi)|\, dz\\
&  \le 2\omega_V(r) \big(|G(y,\xi)|+\omega_V(r)|\xi|\big)
\end{split}
\]
for every $x\in B_r$. Thus 
\[
\ttheta(x,B_r)
= \sup_{\xi \in D(B_r)} \frac{|G(x,\xi)-G_{B_r}(\xi)|}{|G|_{B_r}(\xi)+\omega_V(r)|\xi|} 
\le
2\omega_V(r),
\]
so that $\theta_\gamma(r)\le c \,\omega_V(r)$ for every $\gamma\ge 1$ and so 
$G$ satisfies \VMA{}.

Moreover, if $\omega_V$ satisfies the Dini condition: 
$\int_0^1\omega_V(r)\,\frac{dr}{r}<\infty$, then $G$ satisfies \DMA{\gamma} for any $\gamma\ge 1$. In particular, if $G$ satisfies \VA{} with $\omega(r)=\tilde \omega(r)= r^\alpha$ for some $\alpha>0$, which yields the $C^{1,\alpha}$-regularity (Theorem~\ref{thm:C1alpha}), then $G$ satisfies \DMA{\gamma} for any $\gamma\ge 1$.

If we consider the old version of \VA{} without the $|\xi|$ as described before Definition~\ref{def:continuity}, then we can obtain similar conclusions based on the estimate 
\[
|G(x,\xi)-G_{B_r}(\xi)|
\le 
2\omega_V(r) \big(|G|_{B_r}(\xi)+\omega_V(r)\big) 
\le 
c\, \omega_V(r)^{\frac1{p'}} \big(|G|_{B_r}(\xi)+\omega_V(r)^{\frac1{p'}}|\xi|\big),
\]
provided $G$ satisfies
\azero{} and \ainc{p} with $p>1$ (the inequality is proved like 
Proposition~4.9, \cite{HasO23}, so we omit the details). 
Again, we find that $G$ satisfies \VMA{} but now the modulus of 
continuity is $\omega_V(r)^{1/p'}$. The corresponding Dini condition is 
$\int_0^1\omega_V(r)^{1/p'}\,\frac{dr}{r}<\infty$. This means that the old version does 
not give as precise control of the decay, which is the reason for the 
change to the new version of \VA{}.
\end{example}

We next look at several special cases of the nonlinearity $A(x,\xi)$ in \eqref{maineq} 
to see what the DMO condition entails. 
We will use the $\BMO$-seminorm
\[
\|f\|_{\BMO, r} := 
\sup_{0<\rho\le r} \sup_{B_\rho\subset \R^n} \fint_{B_\rho}|f(x)-(f)_{B_\rho}|\, dx.
\]
With the second supremum, this seminorm is increasing in $r$, so it corresponds 
to $\theta^*$ in Proposition~\ref{prop:meanreverse}; we could also consider a version 
where the supremum is taken only over balls of radius $r$. 
We say that $f$ satisfies VMO (vanishing mean oscillation) 
if $\|f\|_{\BMO, r}\to 0$ when $r\to 0$, and DMO (Dini mean oscillation) if 
$\displaystyle\int_0^1 \|f\|_{\BMO, r} \frac{dr}r<\infty$. 

\begin{example}[Orlicz with coefficient]\label{eg:coefficient}
Let $\phi\in \Phiw$ and $a:\Omega\to [L^{-1}, L]$ for some $L\ge 1$.
Define
\[
A(x,\xi) = D_\xi[a(x)\phi(|\xi|)]=a(x)\frac{\phi'(|\xi|)}{|\xi|}\xi.
\]
Since 
$(A^{(-1)})_{B_r}(\xi) = (a)_{B_r} \phi'(|\xi|)\xi$,  
$|A^{(-1)}|_{B_r}(\xi) = (a)_{B_r} \phi'(|\xi|)|\xi|$
and $a\ge L^{-1}$,
\[
\ttheta(x, B_r)\le 
\frac{|A^{(-1)}(x,\xi) - (A^{(-1)})_{B_r}(\xi)|}{|A^{(-1)}|_{B_r}(\xi)} 
\le L |a(x)-(a)_{B_r}|, \qquad \xi\in\R^n,
\]
and, for any $\gamma\ge 1$,
\[
\fint_{B_r} \ttheta(x,B_r)^\gamma \,dx 
\lesssim
\fint_{B_r}|a(x)-(a)_{B_r}|^\gamma\, dx
\lesssim
\|a\|_{BMO,r}^\gamma.
\]
Therefore, $\theta_\gamma(r)\lesssim \|a\|_{BMO,r}^\gamma$ and so 
$A^{(-1)}$ satisfies \DMA{\gamma} for any $\gamma\ge 1$ or \VMA{} if $a$ satisfies DMO or VMO, respectively.
\end{example}

We next consider double phase energies. In the following example, the mean continuity 
does not give us anything new compared to the point-wise continuity modulus.

\begin{example}[Double phase]\label{eg:double}
Let $\phi(x,t)=t^p+a(x)t^q$, where $1<p\le q$ and $a:\Omega\to [0,L]$, and define
\[
A(x,\xi) := D_\xi \big[\phi(x,|\xi|)\big] = p|\xi|^{p-2}\xi +q a(x) |\xi|^{q-2}\xi.
\]
Then for $\epsilon\ge 0$,
\[\begin{split}
\ttheta_\epsilon(x,B_r) 
:= 
\sup_{\xi \in D_\epsilon(B_r)} \frac{|A^{(-1)}(x,\xi)-(A^{(-1)})_{B_r}(\xi)|}{|A^{(-1)}|_{B_r}(\xi)}
&\le
\sup_{\xi\in D_\epsilon (B_r)} \frac{q|a(x)-(a)_{B_r}|\, |\xi|^{q}}{p|\xi|^{p}}\\ 
&\lesssim
|a(x)-(a)_{B_r}| r^{-\frac{n(q-p)}{p(1+\epsilon)}},
\end{split}\]
where 
$D_\epsilon(B_r) := \{\xi\in \R^n: 
|A^{(-1)}|_{B_r}(\xi)^{1+\epsilon}
\le |B_r|^{-1}\} \subset \{\xi\in \R^n: |\xi|^{p(1+\epsilon)}\le |B_r|^{-1}\}$ 
and we take $\omega\equiv 0$.
Hence
\[
\fint_{B_r}\ttheta_\epsilon (x,B_r)^\gamma\, dx \lesssim r^{-\frac{\gamma n(q-p)}{p(1+\epsilon)}} \fint_{B_r} |a(x)-(a)_{B_r}|^\gamma\,dx \lesssim
 \left(r^{-\frac{ n(q-p)}{p(1+\epsilon)}} \|a\|_{\mathrm{BMO},r} \right)^\gamma.
\]
Now in order for this to be useful, we need that 
$\|a\|_{\mathrm{BMO},r} \lesssim r^{\frac{n(q-p)}{p(1+\epsilon)}}$ for some 
$\epsilon\ge 0$. But then $a$ must be $\frac{n(q-p)}{p(1+\epsilon)}$-Hölder continuous by the Campanato-embedding, so $a$ has point-wise continuity 
modulus and the mean continuity becomes irrelevant.
\end{example}

In the next example, we exchange the coefficients $1$ and $a(x)$ of $|\xi|^p$ and $|\xi|^q$ 
in the double phase problem. The end result has standard $q$-growth, but illustrates the role of 
$\omega$ in Definition~\ref{def:meanoscillation} which is related to small values of $|\xi|$.

\begin{example}
Let $\phi(x,t)=a(x) t^p+t^q$, where $1<p\le q$ and $a:\Omega\to [0,L]$, and define
\[
A(x,\xi) := D_\xi \big[\phi(x,|\xi|)\big] = pa(x)|\xi|^{p-2}\xi +q |\xi|^{q-2}\xi.
\]
Then we calculate, for $\xi\in\R^n$,
\[\begin{split}
\frac{|A^{(-1)}(x,\xi) - (A^{(-1)})_{B_r}(\xi)|}{|A^{(-1)}|_{B_r}(\xi)+\omega(r)|\xi|}
&=
\frac{p|a(x)-(a)_{B_r}| \, |\xi|^{p-1}}{p(a)_{B_r} |\xi|^{p-1} + q|\xi|^{q-1} + \omega(r)} \\
&\approx
|a(x)-(a)_{B_r}| \min\bigg\{\frac{1}{(a)_{B_r}}, |\xi|^{p-q}, 
\frac{|\xi|^{p-1}}{\omega(r)}
\bigg\}.
\end{split}\]
Since $t^{p-q}$ is decreasing and $\frac{t^{p-1}}{\omega(r)}$ is increasing in $t$, the largest value of $\min\{t^{p-q}, \frac{t^{p-1}}{\omega(r)}\}$
occurs when $t^{q-1}=\omega(r)$ and equals $\omega(r)^{(p-q)/(q-1)}$. 
Thus 
\[\begin{split}
\left(\fint_{B_r} \ttheta(x,B_r)^\gamma \,dx \right)^{1/\gamma}
&\lesssim
\omega(r)^{\frac{p-q}{q-1}} \left(\fint_{B_r}|a(x)-(a)_{B_r}|^\gamma\, dx\right)^{1/\gamma}
\approx 
\omega(r)^{\frac {p-q}{q-1}}\|a\|_{\BMO, r}.
\end{split}\]
Thus $\theta_\gamma(r) + \omega(r) \lesssim \omega(r)^{\frac {p-q}{q-1}}\|a\|_{\BMO, r} + \omega(r)
\le \|a\|_{\BMO, r}^{\frac {q-1}{2q-p-1}}$, where the last estimate is obtained by the minimizing 
choice of $\omega(r)$.
Therefore, $A^{(-1)}$ is \DMA{\gamma} for any $\gamma\ge 1$ if the coefficient $a$ is $\frac{q-1}{2q-p-1}$-DMO. 
\end{example}


Next, we consider nontrivial examples whose mean oscillation conditions are weaker than the corresponding point-wise ones. For $f:\Omega\to \R$, we say that $f$ is $\log$-VMO or $\log$-DMO if
\[
\lim_{r\to 0^+} \|f\|_{\BMO, r}\log\tfrac{1}{r} =0 
\quad\text{or}\quad
\int_0^{1} \|f\|_{\BMO, r}\log(\tfrac{1}{r})\, \frac{dr}{r}<\infty,
\]
respectively. If $\|f\|_{\BMO, r}$ is replaced by the point-wise continuity 
modulus $\omega_f$, then we get the stronger 
vanishing $\log$-H\"older continuity and the $\log$-Dini continuity. 
By Lemma~\ref{lem:Spanne}, the $\log$-DMO condition implies the $\log$-H\"older continuity 
$\omega_f(r)\lesssim (\log\frac{1}{r})^{-1}$. 
However, the $\log$-VMO does not imply the $\log$-H\"older continuity.

\begin{example}[Borderline double phase]\label{eg:borderline}
Let $\phi(x,t)= t^p+a(x)\log(1+t)t^p$, where $1<p<\infty$ and $a:\Omega\to [0,L]$, and define
\[
A(x,\xi) := D_\xi \big[\phi(x,|\xi|)\big] =\bigg\{p+ a(x) \bigg(p\log(1+|\xi|)+\frac{|\xi|}{1+|\xi|}\bigg)\bigg\}|\xi|^{p-2}\xi.
\]
For any $\xi\in D(B_r)$, since $|\xi|^{p}\le |B_r|^{-1}$,
\[\begin{split}
\frac{|A^{(-1)}(x,\xi) - (A^{(-1)})_{B_r}(\xi)|}{ |A^{(-1)}|_{B_r}(\xi)}
& =
\frac{|a(x)-(a)_{B_r}| \{p\log(1+|\xi|) + \frac{|\xi|}{1+|\xi|}\}}{p + (a)_{B_r} \{p\log(1+|\xi|) + \frac{|\xi|}{1+|\xi|}\}} \\
&\lesssim |a(x)-(a)_{B_r}| \log(e+|\xi|)\lesssim |a(x)-(a)_{B_r}| \log \tfrac1r.
\end{split}
\]
This implies, for $\gamma\ge 1$, that 
\[
\fint_{B_r} \ttheta(r)^{\gamma}\,dx \lesssim \left(\|a\|_{\mathrm{BMO},r}\log\tfrac1r\right)^{\gamma}.\qedhere
\]
Therefore, $A^{(-1)}$ satisfies \DMA{\gamma} for any $\gamma\ge 1$ or \VMA{} 
if $a$ is $\log$-DMO or $\log$-VMO, respectively.
\end{example}

The Dini mean oscillation condition is inspired by the coefficient case. Consequently, it is 
more difficult to obtain results when the variability is not in a coefficient, as the next 
result illustrates. 

\begin{proposition}[Variable exponent]\label{prop:variable-exponent}
Let $\phi(x,t):=\tfrac{1}{p(x)}t^{p(x)}$, where $p:\Omega\to [p^-,p^+]$ with $1<p^-\le p^+$, and 
define
\[
A(x,\xi) 
:= 
D_\xi \big[\phi(x,|\xi|)\big]= |\xi|^{p(x)-2}\xi.
\]
If $p$ is $\log$-VMO, then $A^{(-1)}$ 
satisfies Definition~\ref{def:meanoscillation} 
for any $\gamma\ge1$ with $\theta(r)=\|p\|_{\mathrm{BMO},r}\log\frac1r$ and $\omega(r)= r^{p^--1}$.
\end{proposition}
\begin{proof}
Let $r\in(0,1]$. Note that $D(B_r)=\{\xi\in\R^n:(|\xi|^{p(\cdot)})_{B_r}
\le |B_r|^{-1}\}\subset \{\xi\in\R^n: |\xi|\in [0,r^{-n}]\}$.
If $|\xi|\in [0, r^{2}]$, 
\[
\frac{|A^{(-1)}(x,\xi) - A_{B_r}^{(-1)}(\xi)|}{|A^{(-1)}|_{B_r}(\xi) + \omega(r)|\xi|} 
=
\frac{|A(x,\xi) - A_{B_r}(\xi)|}{|A|_{B_r}(\xi)+ r^{p^--1}} 
\le
2 r^{p^--1}.
\]
If $|\xi|=1$, $|A(x,\xi) - A_{B_r}(\xi)|=0$ and our estimate will be trivial. 
Suppose $|\xi|\in [r^2,r^{-n}]\setminus \{1\}$ and define 
\[
\bar{p} := - \log_{|\xi|} \fint_{B_r}|\xi|^{-p(y)}\, dy \in [p_{B_r}^-, p_{B_r}^+].
\] 
From $|\xi|^{p(x)-p(y)}\le 
\max\{r^{-n\,|p(x)-p(y)|}, r^{-2\,|p(x)-p(y)|}\} = r^{-n\,|p(x)-p(y)|}$ it follows that 
\[
|\xi|^{p(x)-\bar{p}} 
=
\fint_{B_r}|\xi|^{p(x)-p(y)}\, dy
\le
\fint_{B_r}r^{-n\,|p(x)-p(y)|}\, dy
=
\fint_{B_r} \exp(n\,|p(x)-p(y)|\log\tfrac1r)\, dy.
\]
An analogous argument gives a lower bound with integrand 
$\exp(-n\,|p(x)-p(y)|\log\tfrac1r)$. 
Furthermore, $1-e^{-s}\le e^s - 1$ when $s\ge 0$, and $(\fint_{B_r}|\xi|^{p(y)}\,dy)^{-1}\le |\xi|^{-\bar p}$ by Jensen's inequality.
Thus we obtain that 
\[\begin{split}
\frac{|A^{(-1)}(x,\xi) - A^{(-1)}_{B_r}(\xi)|}{|A^{(-1)}|_{B_r}(\xi)+r^{p^--1}} 
&\le
\frac{||\xi|^{p(x)} - |\xi|^{\bar p}|}{ |\xi|^{\bar p}} 
=
||\xi|^{p(x)-\bar{p}} - 1|\\
&\le 
\fint_{B_r} \big[\exp(n\,|p(x)-p(y)|\log\tfrac1r)-1\big]\, dy.
\end{split}\]
Combining the above results and Hölder's inequality, we have shown that
\[
\fint_{B_r}\ttheta(x,B_r)^\gamma\, dx
\lesssim 
\fint_{B_r}\fint_{B_r} \big[\exp(n\,|p(x)-p(y)|\,\log\tfrac1r)-1\big]^\gamma\, dy\,dx +r^{\gamma(p^--1)}.
\]
By differentiation it follows that 
$e^{ab}-1\le ae^b$ when $b>0$ and $a\in [0, 1-\frac{\ln b}b]$. In 
particular, this holds for all $a\in[0, 1-\frac1e]$.
It follows that
\[\begin{split}
&\fint_{B_r}\fint_{B_r} \big[\exp(n\,|p(x)-p(y)|\,\log\tfrac1r)-1\big]^\gamma\, dy\,dx\\
&\le
(c_n n\gamma \,\|p\|_{\BMO_r}\log\tfrac1r)^\gamma
\fint_{B_r} \fint_{B_r} \exp\bigg(\frac{|p(x)-p(y)|}{c_n \|p\|_{\BMO,r}}\bigg)\, dy\, dx \\
&\lesssim 
(c_n n\gamma \,\|p\|_{\BMO_r}\log\tfrac1r)^\gamma
\fint_{B_r} \exp\bigg(\frac{|p(x)-p_{B_r}|}{c_n \|p\|_{\BMO,r}}\bigg)\, dx ,
\end{split}\]
where $c_n > 1$ is the constant from the John--Nirenberg lemma depending on $n$ and $r>0$ is 
so small that $c_n n\gamma \,\log\tfrac1r \,\|p\|_{\BMO,r}\le 1-\frac1e$ by the $\log$-VMO
assumption. The John--Nirenberg lemma ensures that the integral on the right-hand side is finite. 
Thus we have shown that 
\[
\bigg(\fint_{B_r}\ttheta(x,B_r)^\gamma\, dx \bigg)^\frac1\gamma
\lesssim 
\log\tfrac1r \,\|p\|_{\BMO,r} + r^{p^--1}
\]
for all sufficiently small $r$. For $r\ge r_1$, the estimate still holds since 
$\ttheta$ is bounded by the constant $c r_1^{p^--p^+}$. 
\end{proof}

The definition of $\log$-DMO and $\log$-VMO combined with the previous 
result gives the following example. 

\begin{example}[Variable exponent]
In the setting of Proposition~\ref{prop:variable-exponent}, 
$A^{(-1)}$ satisfies \DMA{\gamma} for any $\gamma\ge 1$ or \VMA{} if $p$ is $\log$-DMO 
or $\log$-VMO, respectively. 
\end{example}

Finally, we show that also the Lorentz-type restriction is a special case of our condition.
Interestingly, this proof does not holds if we include the supremum with respect to the radius in the definition of $\theta_\gamma$, i.e.\ use 
$\theta^*_\gamma$ from Proposition~\ref{prop:meanreverse}. This shows the importance of 
having the optimization in our condition. 
We define the \emph{non-increasing rearrangement $f^*_\Omega$} by
\[
f^*_E(t) 
:= 
\inf\{s\ge 0 \mid \mu_E(s)\le t \}
\quad\text{where}\quad \mu_E(s):=\big|\big\{x\in E \,\big|\, |f(x)|>s\big\}\big|.
\]
Note that $f^*_E$ is a kind of inverse of the distribution function $\mu_E$. 
The definition directly implies that $(|f|^\gamma)^*_\Omega=(f^*_\Omega)^\gamma$. 
The \emph{Lorentz space $L^{n,1}(\Omega)$} is defined by the condition 
$\int_0^\infty t^{1/n} f^*_\Omega(t)\frac{dt}t < \infty$. 

\begin{proposition}[Orlicz with Lorentz coefficient]\label{prop:Lorentz}
Let $A(x,\xi)=a(x)\frac{\psi'(|\xi|)}{|\xi|}\xi$, where $a:\Omega\to [L^{-1},L]$ with $L\ge 1$. If $a\in W^{1,1}(\Omega)$ with $|Da|\in L^{n,1}(\Omega)$, then 
$A^{(-1)}$ satisfies \DMA{\gamma} for any $\gamma\ge 1$.
\end{proposition}
\begin{proof} 
We follow the argument in \cite[Section 2.3]{KuuMin14}.
We consider $\gamma>n'$; the case $\gamma\le n'$ follows from this by Hölder's inequality. 
For $B_{2r}\subset \Omega$ and $\tilde \gamma:=\frac{\gamma n}{n+\gamma}\in (1,n)$, 
by the Sobolev--Poincar\'e inequality,
\[\begin{split}
\bigg(\fint_{B_r}|a-(a)_{B_r}|^\gamma\, dx\bigg)^{\frac1\gamma} 
 \le c r\bigg(\fint_{B_r}|Da|^{\tilde\gamma}\, dx\bigg)^{\frac1{\tilde\gamma}} 
 & \le c r \bigg( \fint^{|B_r|}_0 (|Da|^{\tilde \gamma})^*_{B_r}(t)\, dt\bigg)^{\frac1{\tilde\gamma}}\\
& \le c r \bigg( \fint^{|B_r|}_0 (|Da|^{\tilde \gamma})^*_{\Omega}(t)\, dt\bigg)^{\frac1{\tilde\gamma}}
\end{split}\]
where we used that $f_{U_2}^{*}\le f_{U_1}^{*}$ if $U_2\subset U_1$. 
Note that the right hand side is independent of the center of $B_r$. Taking supremum over $B_r\subset\Omega$ and changing variables $\tau=|B_r|=|B_1|r^n$, we find that 
\[
\int_{0}^{1} \theta_\gamma(r) \,\frac{dr}{r}
\le c \int_0^{1} \bigg( \fint^{|B_r|}_0 (|Da|^{\tilde \gamma})^*_{\Omega}(t)\, dt\bigg)^{\frac1{\tilde\gamma}} \,dr 
\le 
c \int_0^{\infty} \tau^{\frac{1}{n}-1} \bigg( \fint^{\tau}_0 (|Da|^{\tilde \gamma})^*_{\Omega}(t)\, dt\bigg)^{\frac1{\tilde\gamma}} \, d\tau.
\]
We recall a Hardy-type inequality for the quasinorm case $\beta<1$
(see \cite[Theorem~3]{Stepanov93} with $p=q$ and $w(\tau)=v(\tau)=\tau^{\alpha-1}$):
\[
\int_0^\infty \tau^{\alpha-1} \bigg(\fint_0^\tau f(t)\, dt\bigg)^{\beta}\, d\tau \le c(\alpha,\beta) \int_0^\infty t^{\alpha-1} f(t)^{\beta} \, dt, \quad 0<\alpha<\beta<1,
\]
 where $f$ is a non-negative, non-increasing function. Using this inequality with $\beta=\frac1{\tilde \gamma}\in (\frac1n, 1)$ and $\alpha=\frac1n$, we obtain
\[
\int_{0}^{1} \theta_\gamma(r) \,\frac{dr}{r}
\lesssim 
\int_0^{\infty} t^{\frac{1}{n}-1} (|Da|^{\tilde \gamma})^*_{\Omega}(t)^{\frac1{\tilde\gamma}} \, dt
= 
\int_0^{\infty} t^{\frac{1}{n}} (|Da|)^*_{\Omega}(t) \, \frac{dt}t < \infty,
\]
by $(|f|^\gamma)^*_\Omega=(f^*_\Omega)^\gamma$ and the definition of the Lorentz space. 
\end{proof}

We can combine the previous two propositions to cover the variable exponent 
with coefficient and regularity given by Lorentz spaces as considered by Baroni \cite{Bar23}. 

\begin{example}[Variable exponent with Lorentz conditions]
Let $A(x,\xi)=a(x)|\xi|^{p(x)-2}\xi$, where $a:\Omega\to [L^{-1},L]$ with $L\ge 1$. 
If $a, p\in W^{1,1}(\Omega)$ with $|Da|, |Dp|\in L^{n,1}(\Omega)$ and 
$\|Dp\|_{L^{n,1}(B_r)}\log \frac1r\le c$, then 
we can show that $A^{(-1)}$ satisfies \DMA{\gamma} for any $\gamma\ge 1$.

To reach this conclusion, we estimate 
$|A|_{B_r}(\xi)\ge L^{-1}|\xi|^{\bar p-1}$ and 
\[
\big|a(x)|\xi|^{p(x)-1}-a(y)|\xi|^{p(y)-1}\big|
\le
L\big||\xi|^{p(x)-1}-|\xi|^{p(y)-1}\big|
+
|a(x)-a(y)|\, |\xi|^{p(y)-1}.
\]
Since $L^{n,1}(\Omega) \hookrightarrow L^\infty_\loc(\Omega)$, $p$ is $\log$-Hölder 
continuous and so $|\xi|^{p(y)-\bar p}\le c$ \cite[(1.10)]{Bar23}. Then we handle 
$|a(x)-a(y)|$ in the second term 
as in Proposition~\ref{prop:Lorentz} using $|Da|\in L^{n,1}(\Omega)$. For the first term we 
arrive as in Proposition~\ref{prop:variable-exponent} at an estimate 
\[
\log\tfrac1r \, \fint_{B_r}|p(x)-(p)_{B_r}|\, dx, 
\]
from which the integral is estimated as in Proposition~\ref{prop:Lorentz}. The details are left 
to the interested reader.
\end{example}


\section{Comparison estimates}\label{sect:comparison}

In this section, we always assume that $A:\Omega\times \R^n \to \R^n$ satisfies the quasi-isotropic $(p,q)$-growth condition in Definition~\ref{def:quasiisotropy}.

\subsection{Growth functions and approximating energies}
We start with recalling the properties of a so-called growth function $\phi$ of $A$.

\begin{proposition}[Proposition~3.3, \cite{HasO22b}]\label{prop:growthfunction}
There exists $\phi\in \Phic(\Omega)$ with $\phi'(x,\cdot)\in C([0,\infty))$ such that $\phi'$ is \azero{}, \inc{p-1} and \dec{q_1-1} for some $q_1\ge q$ and that 
\begin{equation}\label{phi-growth}
\bar L^{-1}\big(|A(x,\xi)| + |\xi| |D_\xi A(x,\xi)|\big) 
\le \phi'(x,|\xi|)
\le \bar L\, |\xi|\, D_\xi A(x,\xi) e \cdot e
\end{equation}
for every $x\in\Omega$ and $\xi,e\in \R^n$ with $\xi\neq 0$ and $|e|=1$. Here, $q_1$ and $\bar L\ge 1$ depend on $n$, $p$, $q$ and $L$. We call this $\phi$ the \emph{growth function} of $A$.
\end{proposition}

By the monotonicity of the \adec{}-condition, we can replace our original $q$ by $q_1$ from 
the proposition above and assume without loss of generality that $q_1=q$. 

From growth and ellipticity inequalities of growth functions in 
Proposition~\ref{prop:growthfunction}, a standard calculation yields the 
following monotonicity and coercivity/growth properties: 
\[
\big(A(x,\xi_1) - A(x,\xi_2) \big)\cdot (
\xi_1-\xi_2) \gtrsim \frac{\phi'(|\xi_1|+|\xi_2|)}{|\xi_1|+|\xi_2|} |\xi_1-\xi_2|^2
\]
and
\[ 
c_*^{-1} \phi(x,|\xi|) \le \xi \cdot A(x,\xi) \le |\xi| |A(x,\xi)| \le c_* \phi(x,|\xi|)
\]
for some $c_*=c_*(n,p,q,L)\ge1$. In particular, this implies that the function space associated with the local weak solutions to \eqref{maineq} or \eqref{maineq1}
is the Sobolev space $W^{1,\phi}_{\loc}(\Omega)$, and that $\phi$ satisfies \aone{} or \SA{} if and only if $A^{(-1)}$ does. 
When this is the case, we can take $v\in W^{1,\phi}(\Omega)$ with $\mathrm{supp}\, v \Subset \Omega$ as a test function in the weak form of \eqref{maineq1} by approximation (see \cite[Section 4.4]{HarH19}).

We state the higher integrability result for the weak solutions to \eqref{maineq1} which 
requires the \aone{} condition of $A^{(-1)}$. The homogeneous case, when $F\equiv 0$, can be found in \cite[Theorem 4.1]{HasO23}; see also \cite[Lemma 4.7]{HasO22} and references therein. 
Essentially the same argument can be applied to the nonhomogeneous case, so it is not repeated here. 

\begin{lemma}\label{lem:high}
Suppose 
that $A^{(-1)}$ satisfies \aone{}. Let $u\in W^{1,\phi}_{\loc}(\Omega)$ be a weak solution to \eqref{maineq1} with $\phi(\cdot,|F|)\in L^{s}_{\loc}(\Omega)$ for some $s>1$. There exists $\sigma=\sigma(n,p,q,L,s)>0$ such that $\phi(\cdot, |Du|)\in L^{1+\sigma}_{\loc}(\Omega)$ and 
\[
\fint_{B_r} \phi(x,|Du|)^{1+\sigma}\,dx \le c \left\{ \phi_{B_{2r}}^-\left( \fint_{B_{2r}} |Du|\,dx\right)^{1+\sigma} + \fint_{B_{2r}}\phi(x,|F|)^{1+\sigma}\, dx +1\right\}
\]
for some $c=c(n,p,q,L,s)\ge 1$, whenever $B_{2r}\Subset\Omega$ and 
$\varrho_{L^\phi(B_{2r})}(|Du|)\le 1$. 
\end{lemma}
 
We establish an approximating autonomous problem for \eqref{maineq}. 
Denote the averages of $A(\cdot,\xi)$, $\phi(\cdot,t)$ and $\phi'(\cdot,t)$ over $B_r\subset\Omega$ by 
$A_{B_r}$,
$\phi_{B_r}$ and $\phi'_{B_r}$, and set
\begin{equation}\label{def:t0}
t_0:=2(\phi_{B_r})^{-1}(c_*|B_r|^{-1}),
\end{equation}
where $c_*\ge 1$ is as above.
We first define
\[
\psi'(t):=
\begin{cases}
\phi_{B_r}'(t)&\text{if}\ \ t\leq t_0,\\
\phi_{B_r}' (t_0)\, (\frac t{t_0})^{p-1}&\text{if} \ \ t_0< t.
\end{cases}
\]
Note that $\phi'_{B_r}$ is the same as the derivative of $\phi_{B_r}$ and $\psi'$ is continuous since $\phi'$ satisfies \inc{p-1} and \dec{q-1}. 
We consider
\[
\psi(t):= \int_0^t \psi'(s)\, ds.
\]
We see that 
$\psi=\phi_{B_r}$ in $[0, t_0]$ and, since $\phi_{B_r}'$ satisfies \inc{p-1}, $\psi\le \phi_{B_r}$ 
in $[t_0,\infty)$.
Fix $\eta \in C^\infty_0(\R)$ with $\eta\ge0$, $\supp \eta \subset (0,1)$ and $\|\eta\|_1=1$.
We define 
\[
\tphi(t)
:=
\int_0^\infty \psi(s) \eta_{rt}(s-t)\,ds
\quad\text{where}\quad
\eta_r(t):=\tfrac1r \eta(\tfrac tr).
\]
The construction of $\tphi$ is analogous to the one in \cite[Section~5]{HasO22} with $t_1=0$ and $t_2=t_0$, except that 
here we use the average $\phi_{B_r}$ instead of the function at the center-point, $\phi(x_0, \cdot)$. 
Therefore, we have the following analogue of Propositions~5.10 and 5.12 in \cite{HasO22}. 
%
%

\begin{proposition}\label{prop:approxProp}
Suppose 
that 
$\phi\in C^1([0,\infty))\cap C^2((0,\infty))$ satisfies \aone{} and that $\phi'$ satisfies \azero{}, \ainc{p-1} and \dec{q-1} for some $1<p \le q$. 
The following 
hold with $c\ge 1$ depending only on $n$, $p$, $q$ and $L$: 
\begin{enumerate}
\item $\psi(t) \leq \tphi(t)\leq (1+cr)\psi(t)\leq c \psi(t)$ for all $t>0$.
\item $\tphi\in C^1([0,\infty))\cap C^2((0,\infty))$ satisfies \azero{}, \inc{p} and \dec{q} while 
$\tphi'$ satisfies \azero{}, \inc{p-1} and \dec{q-1}.
\item $\tphi(t) \le c (\phi(x,t) +1)$ for all $(x,t)\in B_r\times[0,\infty)$. 
\item Let 
\[
\tilde \zeta(x,t):=
\begin{cases}
\phi(x,\tphi^{-1}(1))t, &0\le t<1,\\
\phi(x,\tphi^{-1}(t)),& t\ge 1,
\end{cases}
\]
and $\zeta(x,t):=\tilde \zeta(x,t)^{1+\tilde \sigma}$ for $\tilde\sigma>0$. Then $\zeta\in \Phiw(B_r)$ 
satisfies \azero{}, \aone{},
\ainc{1+\tilde\sigma} and \adec{q(1+\tilde \sigma)/p} with relevant constants depending on $n,p,q,L$ and $\tilde \sigma$.
\end{enumerate}
\end{proposition}

\begin{proof}
The proofs of (1)\,--\,(3) can be obtained from \cite[Proposition 5.10]{HasO22} when 
we replace $\phi'(x_0,t)$ by $\phi_{B_r}'(t)$. 

Let us prove (4) by following the proof of \cite[Proposition 5.12]{HasO22}.  Note that $\zeta\in \Phiw(B_r)$ is clear once we show 
\ainc{1}. As $\phi$ and $\tphi$ satisfy \azero{}, 
so does $\tilde\zeta$ and thus $\tilde\zeta(x,t)\approx t$ for $t\in [0, 1)$. 
Now we prove that $\tilde \zeta$ satisfies
\ainc{1} and \adec{q/p}, hence $\zeta$ satisfies
\ainc{1+\tilde \sigma} and \adec{q(1+\tilde \sigma)/p}. For $t\in[1,\tphi(t_0)]$, since $\phi^-_{B_{r}}(\tphi^{-1}(t)) \le
\phi_{B_r}(t_0) \lesssim |B_r|^{-1}$, the \aone{} condition of $\phi$ 
and (1) yield 
$\phi(x,\tphi^{-1}(t))\approx \phi_{B_r}(\tphi^{-1}(t))=\psi (\tphi^{-1}(t))\approx t $. 
Therefore $\tilde\zeta(x, t)\approx t$ in $[0,\tphi(t_0)]$. For $t\in [\tphi(t_0),\infty)$, 
setting $s:=\tphi^{-1}(t)$, 
\[
\frac{\tilde\zeta(x,t)}{t} = \frac{\phi(x,\tphi^{-1}(t))}{t} 
\approx 
\frac{\phi(x,s)}{\psi(s)} 
\approx
\frac{t_0^{p-1}}{\phi_{B_r}'(t_0)} \frac{\phi(x,s)}{s^p} 
\]
and, similarly, 
\[
\frac{\tilde\zeta(x,t)}{t^{q/p}} \approx \left(\frac{t_0^{p-1}}{\phi_{B_r}'(t_0)} \right)^{q/p}\frac{\phi(x,s)}{s^q}. 
\]
Therefore, \ainc{p} and \dec{q} of $\phi$ imply 
\ainc{1} and \adec{q/p} of $\tilde\zeta$. Finally, we show that $\zeta$ satisfies \aone{}.
Let $B_{\rho}\subset B_r$, and assume that $\zeta^-_{B_\rho}(t)\le |B_{\rho}|^{-1}$. 
Then 
\[
\tilde \zeta^-_{B_\rho}(t) 
=\phi^-_{B_\rho}(\tphi^{-1}(t)) 
\le |B_{\rho}|^{-1/(1+\tilde \sigma)} \le 
|B_{\rho}|^{-1}.
\]
Therefore, \aone{} of $\phi$ implies that 
\[
\zeta^+_{B_\rho}(t) 
=[\phi^+_{B_\rho}(\tphi^{-1}(t))]^{1+\tilde\sigma} 
\lesssim [\phi^-_{B_\rho}(\tphi^{-1}(t))]^{1+\tilde \sigma}
= \zeta^-_{B_\rho}(t) 
\]
and so $\zeta$ satisfies \aone{}.
\end{proof}

From the nonlinearity $A$, we define an autonomous function $\tilde A:\R^n\to\R^n$ as
\[
\tA(\xi):= \eta_1(|\xi|)A_{B_r}(\xi) + \eta_2(|\xi|)\frac{\phi_{B_r}'(t_0)}{t_0^{p-1}}|\xi|^{p-2}\xi,
\]
where $\eta_1\in C^\infty([0,\infty))$ is such that 
$0\le\eta_1 \le 1$, $\eta_1\equiv 1$ on $[0,t_0]$, 
$\eta_1 \equiv 0$ on $[2t_0,\infty)$ and $-2/t_0 \le \eta_1'\le 0$, and $\eta_2\in C^\infty([0,\infty))$ 
is such that $0\le\eta_2\le1$, $\eta_2\equiv 0$ on $[0,t_0/2]$, $\eta_2\equiv 1$ on $[t_0,\infty)$ and 
$0 \le \eta_2'\le 4/t_0$. Then by the same computations as in 
\cite[Lemma 5.2]{HasO22b} with $t_1=0$, 
$t_2=t_0$ and replacing $\eta_2$ and $\eta_3$ by the functions $\eta_1$ and $\eta_2$ defined above, we can show that $\tphi$ defined above is a growth function of $\tA$, i.e.
\begin{equation}\label{tphi-growth}
\tilde L^{-1}(|\tA(\xi)| + |\xi| |D \tA(\xi)|) \le \tphi'(|\xi|) 
\le \tilde L\, |\xi|\, D \tA(\xi) e \cdot e
\end{equation}
for some $\tilde L\ge 1$ and every $\xi,e\in \R^n$ with $\xi\neq 0$ and $|e|=1$.

\subsection{Approximation and comparison estimates} \label{subsect:approximation}

Let $u \in W^{1,\phi}_{\loc}(\Omega)$ be a local weak solution to \eqref{maineq}. 
Fix $\Omega'\Subset\Omega$. Then, by Lemma~\ref{lem:high}, $\phi(\cdot,|Du|)\in L^{1+\sigma}(\Omega')$. Assume that $r$ satisfies 
\begin{equation}\label{eq:rrestriction1}
r\in (0, \tfrac{1}{2}]
\quad\text{and}\quad
|B_{2r}| \le 2^{-\frac{2(1+\sigma)}{\sigma}}\left(\int_{\Omega'}\phi(x,|Du|)^{1+\sigma}\,dx+1\right)^{- \frac{2+\sigma}{\sigma}}\le \frac{1}{2}.
\end{equation}
By H\"older's inequality this gives  
\[
\int_{B_{2r}} \phi(x,|Du|)^{1+\frac{\sigma}{2}}\,dx 
\le |B_{2r}|^{\frac{\sigma}{2(1+\sigma)}} \left(\int_{\Omega'}\phi(x,|Du|)^{1+\sigma}\,dx\right)^{\frac{2+\sigma}{2(1+\sigma)}} \le \frac{1}{2}
\]
so that 
\[
\varrho_{L^\phi(B_{2r})}(|Du|)
\le \int_{B_{2r}} \phi(x,|Du|)^{1+\frac{\sigma}{2}} \,dx +|B_r| \le 1.
\]
We fix $B_{2r}\Subset \Omega'$, and let $\tu\in u+W^{1,\tphi}_0(B_r)$ be the unique weak solution to
\begin{equation}\tag{$\div\tA$}\label{eqv}
\div \tA(D\tu) =0 
\end{equation}
in $B_r\subset\Omega$. 
The solutions $u$ and $\tu$ of \eqref{maineq} and \eqref{eqv} satisfy the following estimates. 

\begin{lemma}
\label{lem:GreverseDu}
Suppose 
that $A^{(-1)}$ satisfies \aone{} and $r$ satisfies 
\eqref{eq:rrestriction1} and let $\sigma>0$ be from Lemma~\ref{lem:high}.
There exists $C_0\geq 1$ depending only on $n$, $p$, $q$ and $L$ such that 
\begin{enumerate}
\item
$\displaystyle\int_{B_r} \tphi(|D\tu|) \, dx \le C_0 \int_{B_r} \tphi(|Du|) \, dx$,
\smallskip
\item
$\displaystyle\bigg(\fint_{B_r}\phi(x,|Du|)^{1+\sigma}\,dx\bigg)^{\frac{1}{1+\sigma}} 
\le 
C_0\tphi \bigg(\fint_{B_{2r}}|Du|\,dx\bigg)+C_0$,
\smallskip
\item
$\displaystyle\fint_{B_r}|D\tu|\,dx 
\le 
C_0 \fint_{B_{2r}}|Du|\,dx+C_0$,
\smallskip
\item 
if $\varrho_{L^{\phi^{1+\tilde\sigma}}(B_r)}(|Du|)\le M_0<\infty$ for $\tilde\sigma>0$, then
\[
\fint_{B_r}\phi(x,|D\tu|)^{1+\tilde \sigma}\,dx
\le 
C_0 \big(M_0^{\frac{q}{p}-1}+1\big) \left(\fint_{B_{2r}}\phi(x,|Du|)^{1+\tilde\sigma}\,dx+1\right),
\]
where the constant $C_0$ depends also on $\tilde\sigma$.
\end{enumerate}
\end{lemma}
\begin{proof}
Testing \eqref{eqv} by $\tu-u\in W^{1,\tphi}_0(B_{2r})$ and using that 
$\tphi\approx |\tA^{(-1)}|$ by \eqref{tphi-growth} and Young's inequality, we obtain (1). Note that this part does not require \aone{}.

Let $t_1:=\fint_{B_{2r}}|Du|\,dx$. Since $\varrho_{L^\phi(B_{2r})}(|Du|)\le 1$ by the choice of $r$, 
we conclude from Jensen's inequality and \aone{} of $\phi$ that 
\[
t_1 \le
(\phi^-_{B_{2r}})^{-1}\left(\fint_{B_{2r}}\phi^-_{B_{2r}}(|Du|)\,dx\right) \lesssim \phi_{B_r}^{-1}\left(|B_{2r}|^{-1}\right) \lesssim t_0.
\]
Hence, by the definition of $\psi$ and Proposition~\ref{prop:approxProp}(1), $\phi^-_{B_{2r}}(t_1) \le \phi_{B_r} ( t_1) \approx \psi ( t_1) \approx \tphi (t_1)$. 
Then, (2) follows from Lemma~\ref{lem:high}, when $F\equiv 0$, and the fact that $\phi^-_{B_{2r}}(t_1)\lesssim \tphi(t_1)$.
Moreover, using the two estimates we have shown and Proposition~\ref{prop:approxProp}(3), we also obtain
\[
\tphi\left(\fint_{B_r}|D\tu|\,dx \right)
\lesssim
\fint_{B_{r}}\tphi(|Du|)\,dx 
\lesssim 
 \fint_{B_{r}}\phi(x,|Du|)\,dx +1
 \lesssim 
\tphi\left( \fint_{B_{2r}}|Du|\,dx +1\right)
\]
which implies (3).
Finally, (4) follows the Calder\'on--Zygmund estimate in \cite[Lemma~4.15]{HasO22b} with $\phi$ and $\theta$ replaced by $\tphi$ and $\zeta$ given in 
Proposition~\ref{prop:approxProp}(4):
\begin{align*}
\fint_{B_r}\phi(x,|D\tu|)^{1+\tilde\sigma}\,dx
&\lesssim \fint_{B_r}\tilde\zeta(x,\tphi(|D\tu|))^{1+\tilde\sigma}\,dx +1\\
&\lesssim \big(M_0^{\frac{q}{p}-1}+1\big)\left(\fint_{B_r}\tilde\zeta(x,\tphi(|Du|))^{1+\tilde\sigma}\,dx +1\right)\\
&\lesssim \big(M_0^{\frac{q}{p}-1}+1\big) \left(\fint_{B_r}\phi(x,|Du|)^{1+\tilde\sigma}\,dx +1\right). \qedhere
\end{align*}
\end{proof}


We will use the following supremum and excess decay estimates for the 
derivative $D\tu$ of the solution to \eqref{eqv} in $L^1$-space. This is an improvement by Antonini of 
older results from \cite{Lie91,Lie92}. 

\begin{lemma}[Theorem~4.1, \cite{Ant_pp}]\label{lem:v_regularity} Let $\tA:\R^n\to\R^n$ satisfy the quasi-isotropic $(p,q)$-growth condition and $\tphi\in C^1([0,\infty))\cap C^2((0,\infty))$ be its growth function satisfying \eqref{tphi-growth}. 
Let $\tu\in W^{1,\tphi}(B_r)$ be a weak solution to \eqref{eqv}.
There exist $\alpha\in (0,1)$ and $C_1 \ge 1$ depending on 
$n$, $p$, $q$ and $\tilde L$ such that, for any $B_{\nu}(y)\subset B_\rho(y)\subset B_r$, 
\[
\| D\tu\|_{L^\infty(B_{\rho/2}(y))} \le 
C_1 \fint_{B_\rho(y)} |D\tu| \,dx, 
\]
and
\[
\fint_{B_{\nu}(y)} |D\tu-(D\tu)_{B_\nu(y)}| \,dx 
\le \osc_{B_\nu(y)} D\tu 
\le 
C_1 \Big(\frac \nu\rho\Big)^{\alpha} \fint_{B_\rho(y)} |D\tu-(D\tu)_{B_\rho(y)}| \,dx.
\]
\end{lemma}

Now we derive comparison estimates between the gradients of $u$ and $\tu$, in terms of 
the mean oscillation conditions in Definition~\ref{def:meanoscillation}.
The following first one corresponds to Lemma 3.1 in \cite{Bar23}. Note that the resulting estimates are almost the same. However, the proof of next lemma requires more delicate analysis since we consider general structure.
A similar inequality was 
proved in \cite[Lemma~6.2]{HasO22}, but with the much worse power 
$\omega^{p/q}$, even though there $\omega$ was a point-wise modulus of continuity, not mean continuity 
like here. 

\begin{lemma} \label{lem:comparison} 
In the setting of Lemma~\ref{lem:GreverseDu} with $\varrho_{L^{\phi^{\gamma/(\gamma-2)}}(B_r)}(|Du|)\le 1$ for some $\gamma>2$, 
\[
\fint_{B_r} \frac{\tphi'(|Du|+|D\tu|)}{|Du|+|D\tu|}|Du-D\tu|^2\,dx 
\le 
c\,\Theta(r)^2 \bigg(\fint_{B_r} \phi(x,|Du|)^{\frac\gamma{\gamma-2}}\,dx+1\bigg)^{\frac{\gamma-2}\gamma}
\]
for some $c=c(n,p,q,L,\gamma)\ge 1$, 
where $\Theta(r):= \omega(r)+\theta_\gamma(r) + r^{\alpha_1}$ with $\alpha_1=\frac{4n}{\gamma(\gamma-2)}$.
\end{lemma}
\begin{proof}
Denote $\Psi:=\frac{\tphi'(|Du|+|D\tu|)}{|Du|+|D\tu|}|Du-D\tu|^2$. 
By the monotonicity of $\tA$ and testing with $u-\tu\in W^{1,\phi}_0 (B_r)\subset W^{1,\tphi}_0 (B_r)$ in the weak formulations of \eqref{maineq} and \eqref{eqv}, we have
\[\begin{split}
\fint_{B_r} \Psi\,dx
& \le c_0 \fint_{B_r} \big(\tA(Du) - \tA( D\tu)\big)\cdot (Du-D\tu)\, dx \\
& = c_0 \fint_{B_r}\big(\tA(Du) - A( x,Du)\big)\cdot (Du-D\tu)\, dx\\
& \le c_0 \fint_{B_r} \underbrace{\big|A(x,Du)-\tA(Du)\big| \, |Du-D\tu|}_{=:\Delta}\, dx.
\end{split}\]

In the set $E:= \{x\in B_r : |A^{(-1)}|_{B_r}(Du(x))
\leq |B_r|^{-1}\}$, since $\phi_{B_r}(|Du|) \le c_* |A^{(-1)}|_{B_r}(Du)
\le c_*|B_r|^{-1}$, it follows from \eqref{def:t0} that $|Du|\le \frac{t_0}2$ and so 
$\tA(Du)=A_{B_r}(Du)$. Hence by 
the definition of $\ttheta$ (Definition~\ref{def:meanoscillation} with $G=A^{(-1)}$) and the fact $\phi_{B_r}(|Du|)= \psi(|Du|)\le \tphi(|Du|)$ from Proposition~\ref{prop:approxProp}(1), 
we conclude that 
\[
\begin{split}
|\tA(Du)-A(x,Du)|
&\le \ttheta (x,B_r) (|A|_{B_r}(Du) + \omega(r))\\
&\lesssim \ttheta (x,B_r) ( \tphi'(|Du|+|D\tu|) + \omega(r)) \\
&\lesssim
\ttheta (x,B_r) \tphi(|Du|+|D\tu|)^{\frac{1}{2}} \bigg[\frac{\tphi'(|Du|+|D\tu|)}{|Du|+|D\tu|} \bigg]^{\frac{1}{2}}
+ \ttheta (x,B_r)\omega(r).
\end{split}
\]
It then follows from the Cauchy--Schwartz inequality that 
\[ \begin{split}
\Delta \chi_E 
&\le 
c \, \ttheta (x,B_r)^2 \tphi(|Du|+|D\tu|)
+ \frac1{2c_0} \Psi + \big(\ttheta (x,B_r)^2 + \omega(r)^2\big)|Du-D\tu|.
\end{split}\]
Integrating the previous inequality over $B_r$, we obtain that
\[ \begin{split}
\fint_{B_r}\Delta\,\chi_{E}\,dx 
\le
c\fint_{B_r} \big[\ttheta (x,B_r)^2+\omega(r)^2\big] \tphi(|Du|+|D\tu|+1)\, dx
+ \frac1{2c_0} \fint_{B_r}\Psi\, dx.
\end{split}\]
The first integral is estimated using H\"older's inequality, the definition 
of $\theta_\gamma$, Proposition~\ref{prop:approxProp}(3) and  Lemma~\ref{lem:GreverseDu}(4) with $\varrho_{L^{\phi^{\gamma/(\gamma-2)}}(B_r)}(|Du|)\le 1$:
\[
\begin{split}
&\fint_{B_r} \left[\ttheta (x,B_r)^2+\omega(r)^2\right] \tphi(|Du|+|D\tu|+1)\,dx \\
&\quad\le 
\bigg\{\bigg(\fint_{B_r} \ttheta (x,B_r)^{\gamma} \,dx\bigg)^{\frac{2}{\gamma}} + \omega(r)^2\bigg\} \bigg(\fint_{B_r} \tphi(|Du|+|D\tu|+1)^{\frac\gamma{\gamma-2}}\,dx\bigg)^{\frac{\gamma-2}\gamma} \\
&\quad\le 
c\, \big(\theta_\gamma(r)^2 +\omega(r)^2\big)\bigg(\fint_{B_r} \phi(x,|Du|)^{\frac\gamma{\gamma-2}}\,dx+1\bigg)^{\frac{\gamma-2}\gamma}. 
\end{split}\]
Hence we obtain that
\[
\fint_{B_r} \Delta\,\chi_{E}\,dx 
\le 
c \big(\theta_\gamma(r)^2+\omega(r)^2\big) \bigg(\fint_{B_r} \phi(x,|Du|)^{\frac\gamma{\gamma-2}}\,dx+1\bigg)^{\frac{\gamma-2}\gamma} 
+ \frac1{2c_0} \fint_{B_r} \Psi\,dx. 
\]

In the set $B_r\setminus E$, $\phi_{B_r}(|Du|)\gtrsim |B_{r}|^{-1}$. 
This implies $\phi^{-}_{B_r}(|Du|)\gtrsim |B_{r}|^{-1}$ since if $\phi^{-}_{B_r}(|Du|)\le |B_{r}|^{-1}$, by \aone{} of $\phi$, $|B_r|^{-1} \lesssim \phi_{B_r}(|Du|) \lesssim \phi^{-}_{B_r}(|Du|) +1 \lesssim \phi^{-}_{B_r}(|Du|)$. In the last estimate we used the fact that $1\lesssim |Du| \lesssim \phi^{-}_{B_r}(|Du|)$ in $B_r\setminus E$.
Furthermore, by 
Proposition~\ref{prop:approxProp}(3) and \azero{}, 
\[
\tphi'(|Du|) 
\approx \frac{\tphi(|Du|)}{|Du|}
\lesssim \frac{\phi(x,|Du|)+1}{|Du|}
\approx
\phi'(x,|Du|). 
\]
Using the growth conditions \eqref{phi-growth} and \eqref{tphi-growth},
the above inequalities 
and Young's inequality twice, we conclude that
\[\begin{aligned}
\Delta\,\chi_{B_r\setminus E}
& \lesssim \big[|B_{r}|\phi^-(|Du|)\big]^{\tilde\sigma} \phi'(x,|Du|) (|Du|+|D\tu|)\\
&\lesssim
r^{n\tilde\sigma} \big(\phi(x,|Du|)^{1+\tilde\sigma} + \phi(x,|Du|)^{\tilde\sigma} \phi(x,|D\tu|)\big)\\
&\lesssim
r^{n\tilde\sigma} \big(\phi(x,|Du|)^{1+\tilde\sigma}+ \phi(x,|D\tu|)^{1+\tilde\sigma}\big), 
\end{aligned}\]
where $\tilde\sigma := \frac{2}{\gamma-2}$. 
From Lemma~\ref{lem:GreverseDu}(4) with $\varrho_{L^{\phi^{\gamma/(\gamma-2)}}(B_r)}(|Du|)\le 1$ it follows that 
\[\begin{split}
\fint_{B_r} \Delta\,\chi_{B_r\setminus E}\, dx 
&\lesssim r^{n\tilde\sigma} 
\bigg(\fint_{B_r}\big[\phi(x,|Du|)^{1+\tilde\sigma} + \phi(x,|D\tu|)^{1+\tilde\sigma}\big]\,dx\bigg)^{\frac{1}{1+\tilde\sigma}+\frac{\tilde \sigma}{1+\tilde\sigma}}\\
&\lesssim r^{n(\tilde\sigma -\frac{\tilde\sigma}{1+\tilde\sigma})} \bigg( \fint_{B_{r}}\phi(x,|Du|)^{1+\tilde\sigma}\,dx+1\bigg)^{\frac{1}{1+\tilde \sigma}}. 
\end{split}\]
Since $1+\tilde\sigma=\frac\gamma{\gamma-2}$, we obtain the desired inequality 
by combining the estimates in $E$ and $B_r\setminus E$ and absorbing 
the integral over $\Psi$ into the left-hand side.
\end{proof}

Applying the same approach as in \cite[Corollary~6.3]{HasO22}, 
we obtain the following rough comparison estimates from the previous lemma with 
an unwanted exponent $\frac1q$ on $\Theta$.

\begin{lemma} \label{lem:comparison1} 
In the setting of Lemma~\ref{lem:comparison}, we further assume that 
\[
\bigg(\fint_{B_r} \phi(x,|Du|)^{\frac\gamma{\gamma-2}}\,dx\bigg)^{\frac{\gamma-2}\gamma} \le C^*_\gamma \,\tphi \bigg(\fint_{B_{2r}} |Du|\,dx+1\bigg),
\]
for some $C^*_\gamma>0$. Then there exists 
$C_2=C_2(n,p,q,L,\gamma,C^*_{\gamma})\ge 1$ such that 
\[\begin{split}
\fint_{B_r} |Du-D\tu|\,dx & \le \tphi^{-1}\bigg( \fint_{B_r} \tphi( |Du-D\tu| ) \,dx\bigg) 
 \le 
C_2 \, \Theta(r)^{\frac{1}{q}} \bigg(\fint_{B_{2r}} |Du| \,dx+1 \bigg).
\end{split}\]
\end{lemma}

The next result is a sharper version of the previous lemma with better exponent of $\Theta$. This is analogous to Lemma 3.5 in \cite{Bar23}, but much simpler since we consider a single homogeneous equation. 

\begin{lemma} \label{lem:comparison2}
In the setting of Lemma~\ref{lem:comparison1}, there exists 
$C_3=C_3(n,p,q,L,\gamma,C^*_{\gamma})\ge 1$ such that 
if
\[
\frac{\lambda}{\Lambda} \le \inf_{B_\rho} |D\tu| 
\quad \text{and} \quad 
\fint_{B_{2r}} |Du| \, dx \le \lambda
\]
for some $\lambda, \Lambda\ge 1$ and $B_\rho\subset B_r$, then 
\[
\fint_{B_\rho} |Du-D\tu| \,dx \le C_3 \Lambda^{q-1} \Big(\frac r\rho\Big)^n \Theta(r) \lambda.
\]
\end{lemma}
\begin{proof}
By the Cauchy--Schwartz inequality, 
\[\begin{split}
\fint_{B_\rho}|Du-D\tu|\,dx
\le
\Theta(r) \fint_{B_\rho}(|Du|+|D\tu|)\,dx + 
\frac1{\Theta(r)}\fint_{B_\rho} \frac{|Du-D\tu|^2}{|Du|+|D\tu|}\,dx.
\end{split}\]
By Lemma~\ref{lem:GreverseDu}(3),
the first integral on the right-hand side is bounded by $c(\frac r\rho)^n \lambda$. 
For the second integral, we use that $\tphi'$ is increasing and satisfies \dec{q-1}
and that $\frac\lambda\Lambda \le |D\tu|+|D u|$ in $B_\rho$, 
to estimate by Lemma~\ref{lem:comparison} and the assumption of Lemma~\ref{lem:comparison1} that
\begin{align*}
\Big(\frac \rho r\Big)^n \fint_{B_\rho} \frac{|Du-D\tu|^2}{|Du|+|D\tu|}\,dx
& \le
\frac{1}{\tphi'(\lambda/\Lambda)} \fint_{B_{r}} \frac{\tphi'(|Du|+|D\tu|)}{|Du|+|D\tu|}|Du-D\tu|^2\,dx\\
& \lesssim \Lambda^{q-1}\Theta(r)^2 \frac{1}{\tphi'(\lambda)} \tphi\bigg(\fint_{B_{2r}}|Du|\,dx\bigg) \\
& \approx \Lambda^{q-1} \Theta(r)^2 \lambda. \qedhere
\end{align*}
\end{proof}


\section{Calder\'on--Zygmund estimates}\label{sect:proofCZ}

In this section, we prove the Calder\'on--Zygmund estimates of Theorem~\ref{thm:CZ}. 
We assume throughout that $A:\Omega\times \R^n\to \R^n$ satisfies the quasi-isotropic $(p,q)$-growth condition in 
Definition~\ref{def:quasiisotropy} and \aone{}.
Let $u\in W^{1,\phi}_{\loc}(\Omega)$ be a weak solution to \eqref{maineq1}. We start with deriving comparison estimates.
Note that by Lemma~\ref{lem:high}, $\phi(\cdot,|Du|)\in L^{1+\sigma}_{\loc}(\Omega)$. Fix $B_{2r} \subset \Omega' \Subset \Omega$ and let $u_r \in u+W^{1,\phi}_0(B_{2r})$ be the unique weak solution to 
\eqref{maineq} in $B_{2r}$.
Note that the existence and uniqueness of the weak solution $u_r$ follow from standard arguments for monotone operators (see, e.g., \cite[Section 8]{HHK16}).
Then we have the following energy and higher integrability estimates for $Du_r$. 

\begin{lemma}\label{lem:CZstart}
In the setting of Proposition~\ref{prop:growthfunction}, 
\[
\int_{B_{2r}} \phi(x,|Du_r|) \, dx \le c \int_{B_{2r}} \phi(x,|Du|) \, dx.
\]
Suppose also that $A^{(-1)}$ satisfies \aone{}. Then there exist $C_4\ge 1$ and $\sigma_1\in (0,\sigma)$ depending on $n$, $p$, $q$ and $L$ such that 
\[
\fint_{B_{2r}} \phi(x,|Du_r|)^{1+\sigma_1} \, dx \le C_4 \left(\fint_{B_{2r}} \phi(x,|Du|)^{1+\sigma_1} \, dx +1 \right).
\]
\end{lemma}
\begin{proof}
The proofs of the estimates are quite standard; hence we only briefly sketch them. By testing \eqref{maineq}
with $u-u_r\in W^{1,\phi}_0(B_{2r})$ and using the coercivity property of growth functions and Young's inequality we obtain the first claim. 
We next prove the second estimate. From Lemma~\ref{lem:high},
\[
 \fint_{B_{r}} \phi(x,|Du_r|)^{1+\sigma} \, dx \le c \bigg\{\bigg(\fint_{B_{2r}} \phi(x,|Du_r|) \, dx\bigg)^{1+\sigma} +1 \bigg\}.
 \]
Using a similar approach as in the proof of Lemma~\ref{lem:high} (see, e.g., \cite[Lemma 3.5]{BO16}), we also find the following higher integrability on the boundary:
 \[\begin{split}
 \fint_{ B_{r}(y)\cap B_{2r}} & \phi(x,|Du_r|)^{1+\sigma_1} \, dx \\
& \le c \bigg\{\bigg(\fint_{B_{2r}(y)\cap B_{2r}} \phi(x,|Du_r|) \, dx\bigg)^{1+\sigma_1} + \fint_{B_{2r}(y)\cap B_{2r}} \phi(x,|Du|)^{1+\sigma_1} \, dx+1 \bigg\},
 \end{split}\]
where $y\in \partial B_{2r}$. The standard covering argument with the previous estimates 
implies the second claim. 
\end{proof}

Now, suppose that $r>0$ satisfies the following stronger version of \eqref{eq:rrestriction1}:
\begin{equation}\label{eq:rrestriction2}
r\le \frac{1}{2}
\quad\text{and}\quad
|B_{2r}| 
\le 
\bigg\{4^{\frac{1+\sigma_1}{\sigma_1}}C_4\bigg(\int_{\Omega'}\phi(x,|Du|)^{1+\sigma_1}\,dx+1\bigg)^{ \frac{2+\sigma_1}{\sigma_1}}\bigg\}^{-1} \le \frac{1}{4C_4}.
\end{equation}
Analogously as \eqref{eq:rrestriction1}, this and Lemma~\ref{lem:CZstart} imply that 
\[
\varrho_{L^\phi(B_{2r})}(|Du|) \le 1
\quad\text{and}\quad
\int_{B_{2r}} \phi(x,|Du_r|)^{1+\frac{\sigma_1}{2}} \, dx
\le \frac{1}{2}.
\]
Moreover, Lemma~\ref{lem:GreverseDu}(2) for $u_r$ and $\sigma_1\in(0,\sigma)$ imply that 
\[
\bigg(\fint_{B_{r}} \phi(x,Du_r)^{1+\frac{\sigma_1}{2}} \, dx\bigg)^{\frac{1}{1+\sigma_1/2}} \le c \tphi \bigg(\fint_{B_{2r}} |Du_r| \, dx +1\bigg)
\]
for some $c>0$ depending on $n$, $p$, $q$ and $L$. Therefore, $u_r$ satisfies the inequalities required in Lemma~\ref{lem:comparison1} with $\gamma=2+\tfrac{4}{\sigma_1}$. 
Let $\tu_r\in u_r+W^{1,\tphi}_0(B_{r})$ be the unique weak solution to
\eqref{eqv}
where $\tA$ and $\tphi$ are defined in Section~\ref{subsect:approximation}. 
Then Lemma~\ref{lem:comparison1} gives the following comparison estimates:

\begin{lemma}\label{lem:approximation} 
In the setting of Lemma~\ref{lem:GreverseDu},
\[
\fint_{B_{r}} \phi(x,|D\tu_r|) \,dx \le c \bigg(\fint_{B_{2r}} \phi(x,|Du|)\, dx +1\bigg),
\]
and, for any $\epsilon\in(0,1)$,
\[
\fint_{B_{r}} \phi(x,|Du -D\tu_r|)\, dx
 \le c( \epsilon + \Theta(r)^{\alpha_2}) \bigg(\fint_{B_{2r}} \phi(x,|Du|) \, dx + 1 \bigg) + \epsilon^{-\frac{p+1}{p}} \fint_{B_{2r}} \phi(x,|F|)\, dx
\]
for some $c>0$ and $\alpha_2\in(0,1)$ depending on $n,p,q$ and $L$. 
\end{lemma}

\begin{proof}
We first obtain a comparison estimate between $Du$ and $Du_r$. 
We take $u-u_r \in W_0^{1,\phi}(B_{2r})$ as a test function in the weak formulations of \eqref{maineq1} and \eqref{maineq} to discover that 
\[
\fint_{B_{2r}} (A(x, Du)-A(x,Du_r)) \cdot (Du-Du_r) \,dx = \fint_{B_{2r}} A(x, F) \cdot (Du-Du_r) \,dx.
\]
By \cite[Proposition~3.8(3)]{HasO22}, the monotonicity and coercivity properties of 
growth functions, the equation above, and Lemma~\ref{lem:CZstart},
we see that 
\begin{align*}
&\fint_{B_{2r}} \phi(x,|Du -Du_r|) \,dx \\
& \lesssim \epsilon \fint_{B_{2r}} \phi(x,|Du|) +\phi(x, |Du_r|) \,dx+ \epsilon^{-1}\fint_{B_{2r}} \frac{\phi'(x,|Du|+|Du_r|)}{|Du|+|
Du_r|} |Du -Du_r|^2 \,dx\\
& \lesssim \epsilon \fint_{B_{2r}} \phi(x,|Du|) \,dx+ \epsilon^{-1}\fint_{B_{2r}} \phi'(x,|F|) |Du -Du_r| \,dx.
\end{align*} 
Then, using Young's inequality we have 
\begin{equation}\label{eq:approximation_pf1}
\fint_{B_{2r}} \phi(x,|Du -Du_r|)\,dx 
\lesssim \epsilon \fint_{B_{2r}} \phi(x,|Du|)\,dx + \epsilon^{-\frac{p+1}{p}} \fint_{B_{2r}} \phi(x,|F|)dx.
\end{equation}

Let $\tau>0$ be the solution of $\frac{\tau}{1+\sigma_1/2}+q(1-\tau)=1$.
By H\"older inequality with exponents $\frac{1+\sigma_1/2}{\tau}$ and $\frac{1}{q(1-\tau)}$ 
applied to $ \phi=\phi^{\tau}\phi^{1-\tau}$, we derive that 
\[\begin{split}
&\fint_{B_{r}} \phi(x,|Du_r-D\tu_r|)\,dx \\
&\le \bigg( \fint_{B_{r}} \phi(x,|Du_r-D\tu_r|)^{1+\frac{\sigma_1}{2}}\,dx\bigg)^{\frac{\tau}{1+\sigma_1/2}} \bigg( \fint_{B_{r}} \phi(x,|Du_r-D\tu_r|)^{\frac{1}{q}}\,dx\bigg)^{q(1-\tau)}. 
\end{split}\]
By Lemma~\ref{lem:GreverseDu}(4) and Lemma~\ref{lem:high} for $u_r$ with $\sigma=\sigma_1$  
we have 
\[\begin{split}
\bigg(\fint_{B_{r}} \phi(x,|Du_r-D\tu_r|)^{1+\frac{\sigma_1}{2}}\,dx\bigg)^{\frac1{1+\sigma_1/2}} 
&\lesssim \bigg(\fint_{B_{r}} \phi(x,|Du_r|)^{1+\frac{\sigma_1}{2}}\,dx\bigg)^{\frac1{1+\sigma_1/2}} +1 \\
&\lesssim \fint_{B_{2r}} \phi(x,|Du_r|) \, dx +1 .
\end{split}\]
Moreover, from
by Jensen's inequality, Lemma~\ref{lem:comparison1} for  $u$ with $\gamma=2+\frac{4}{\sigma_1}$ 
and \aone{} with the fact that $\fint_{B_{2r}} |Du_r|\,dx \lesssim (\phi^-_{B_{2r}})^{-1}(|B_{2r}|^{-1})$, we obtain that 
\begin{align*} 
&\bigg(\fint_{B_{r}} \phi(x,|Du_r-D\tu_r|)^{\frac{1}{q}}\,dx\bigg)^{q}
\le \bigg(\fint_{B_{r}} \phi^+_{B_r}(|Du_r-D\tu_r|)^{\frac{1}{q}}\,dx\bigg)^{q} \\
&\qquad\qquad \lesssim \phi^+_{B_{r}} \bigg(\fint_{B_{r}} |Du_r-D\tu_r|\,dx\bigg)
 \lesssim \Theta(r)^{\frac{p}{q}} \phi^+_{B_{r}} \bigg(\fint_{B_{2r}} |Du_r| \,dx +1\bigg) \\
&\qquad\qquad\lesssim \Theta(r)^{\frac pq} \phi^-_{B_{2r}} \bigg( \fint_{B_{2r}} |Du_r| \,dx+1\bigg) \lesssim \Theta(r)^{\frac pq} \bigg( \fint_{B_{2r}} \phi(x,|Du_r|) \,dx +1\bigg) .
\end{align*}
Combining the previous three estimates yields
\[
\fint_{B_{r}} \phi(x,|Du_r-D\tu_r|)\,dx \lesssim \Theta(r)^{(1-\tau)\frac pq}\left(\fint_{B_{2r}} \phi(x, |Du_r|) \,dx +1\right).
\]
We obtain the desired estimates from this, \eqref{eq:approximation_pf1} and Lemma~\ref{lem:CZstart}.
\end{proof}

Now, we are ready to prove the Calder\'on--Zygmund estimate.

\begin{proof}[Proof of Theorem~\ref{thm:CZ}]
Since we have obtained the relevant comparison estimates, we can prove the theorem by using the well-known argument mentioned in Section~\ref{sect:intro}. Therefore, we will provide a sketch of the proof and refer to, for instance, our recent paper \cite{LOP24} for details. We first observe that by Proposition~\ref{prop:meanreverse},
\[
\lim_{r\to 0^+}\theta_\gamma(r)=0 \quad \text{for }\ \gamma = 2+\tfrac{4}{\sigma_1}.
\]

Let $\delta \in (0,1)$ be a small constant to be determined later, 
and $B_{2r} \subset \Omega'\Subset \Omega$, where $r>0$ is a small number
which satisfies \eqref{eq:rrestriction2} and $\Theta(r)^{\alpha_2}\le\epsilon:=\delta^{\frac{p}{2p+1}{p}}$ 
with notation from Lemma~\ref{lem:approximation}. Set
\begin{equation}\label{def_lambda0}
 \lambda_0 := \fint_{B_{2r}} \phi(x,|Du|)\,dx + \frac{1}{\delta} \fint_{B_{2r}} \phi(x,|F|)\,dx +1
\end{equation}
and for $\rho\in (0, 2r]$ and $\lambda>0$
\[
E(\lambda, \rho)=E(\lambda, B_\rho) : = \{ x \in B_\rho : \phi( x, |Du|) > \lambda\}.
\]

By the Vitali covering lemma, for a given 
$\lambda \ge \alpha \lambda_0,$ where $\alpha : =( \frac{20}{\tau_2 - \tau_1})^n,$ and $1\le \tau_1<\tau_2\le 2$, there exists a disjoint family 
of balls $\{ B_{\rho_i}(y^i)\}_{i=1}^{\infty}$ with $y^i \in E(\lambda, \tau_1 r)$ and $\rho_{i} \in (0, \frac{(\tau_2-\tau_1)r}{10})$ such that
$$E(\lambda, \tau_1 r) \subset \bigcup_{i=1}^{\infty} B_{5\rho_i}(y^i), $$
\[
\fint_{B_{\rho_i}(y^i)} \phi(x, |Du|) \,dx +
\frac{1}{\delta} \fint_{B_{\rho_i}(y^i)} \phi(x, |F|) \,dx = \lambda,
\]
and, for any $ \rho \in ( \rho_i, (\tau_2-\tau_1)r]$,
\[
\fint_{B_{\rho}(y^i)}\phi(x,|Du|) \,dx+ \frac{1}{\delta} \fint_{B_{\rho}(y^i)} \phi(x, |F|) \,dx <\lambda.
\]
Note that the previous inequality can be applied when $\rho=10\rho_i$ and 
that the penultimate display yields
\begin{equation}\label{covering_est}
\frac\lambda2 \left|B_{\rho_i}(y^i)\right|
\le
 \int_{B_{\rho_i}(y^i)\cap\left\{ \phi(x,|Du|) >\frac{\lambda}{4} \right\}}\phi(x,|Du|) \,dx+ \frac1{\delta} \int_{B_{\rho_i}(y^i)\cap\left\{\phi(x,|F|)>\frac{\delta\lambda}{4} \right\}} \phi(x,|F|) \,dx.
\end{equation}

We use a two-step approximation approach in order to be able to apply our earlier results. 
For each $i$, we consider the unique solution $h_r^i \in u+W^{1,\phi}_0(B_{10\rho_i}(y^i))$ to 
\eqref{maineq} in $B_{10\rho_i}(y^i)$
and the unique weak solution
$\tu_r^i \in h_r^i+ W_0^{1,\tphi}(B_{5\rho_i}(y^i))$ to \eqref{eqv} in $B_{5\rho_i}(y^i)$.
Then by Lemma~\ref{lem:approximation} and 
$\Theta(r)^{\alpha_2}\le\epsilon$, we have 
\begin{equation}\label{est_DuDviLambda}
\fint_{B_{5\rho_i}(y^i)} \phi(x, |Du-D\tu_r^i|)\,dx 
\lesssim \epsilon \lambda.
\end{equation}
%
%
Moreover, by Lemma~\ref{lem:GreverseDu}(3), Jensen's inequality, Lemma~\ref{lem:CZstart} and 
$\varrho_{L^\phi(B_{2r})}(|Du|)\le 1$
 we conclude that
\[\begin{split}
\phi^{-}_{B_{10\rho_i}(y^i)}\bigg(\fint_{B_{5\rho_i}(y^i)} |D\tu_r^i|\, dx \bigg)
&\lesssim \phi^{-}_{B_{10\rho_i}(y^i)}\bigg(\fint_{B_{10\rho_i}(y^i)} |Dh_r^i|\, dx+1\bigg) \\
&\lesssim \fint_{B_{10\rho_i}(y^i)}\phi(x, |Dh_r^i|)\, dx +1 
\lesssim |B_{10\rho_i}|^{-1}.
\end{split}\]
Hence by Lemma~\ref{lem:v_regularity}, \aone{} and Jensen's inequality, we have 
\begin{align*}
& \phi\big(y, \| D\tu_r^i\|_{L^{\infty}(B_{5\rho_i}(y^i))}\big) \le c \phi\bigg(y, \fint_{B_{10\rho_i}(y^i)} |D\tu_r^i| \,dx\bigg)\\
&\quad \le c \phi^-_{B_{5\rho_i}(y^i)}\bigg( \fint_{B_{10\rho_i}(y^i)} |D\tu_r^i| \,dx +1\bigg) \le c \fint_{B_{10\rho_i}(y^i)} \phi(x, |D\tu_r^i|) \,dx +1 \le c_0 \lambda
\end{align*}
for some constant $c_0 \ge 1.$ 
 
Set $K : = 2^{q} c_0$ 
and let $y \in B_{5\rho_i}(y^i)$ with $\phi(y, |Du|)> K\lambda$.
Then
\[
\phi(y, |D\tu_r^i|)\le c_0\lambda = 2^{-q} K \lambda < 2^{-q}\phi(y, |Du|)\le 
\phi(y, \tfrac12 |Du|). 
\]
Therefore $|D\tu_r^i|\le \frac12 |Du|$ so that $|Du|\le 2\,|Du-D\tu_r^i|$, 
%
which implies that 
\[
\phi(y, |Du|) \le 2^{q}\phi(y, |Du-D\tu_r^i|).
\]
By \eqref{covering_est} and \eqref{est_DuDviLambda}, we then derive that 
\[
\begin{split}
& \int_{B_{5\rho_i}(y^i) \cap E(K\lambda , \tau_1 r)} \phi(x, |Du|)\, dx\\
 &\le 2^{q} \int_{B_{5\rho_i}(y^i)} \phi(x, |Du-D\tu_r^i|) \,dx\le c \epsilon \lambda \left| B_{\rho_i}(y^i)\right|\\
 & \le c \epsilon \bigg(\int_{B_{\rho_i}(y^i)\cap\left\{ \phi(x,|Du|) >\frac{\lambda}{4} \right\}}\phi(x,|Du|) \,dx+ \frac1{\delta} \int_{B_{\rho_i}(y^i)\cap\left\{\phi(x,|F|)>\frac{\delta\lambda}{4} \right\}} \phi(x,|F|) \,dx\bigg).
 \end{split}
\]
Note that the balls $B_{\rho_i}(y^i)$ are mutually disjoint and
$$
E(K\lambda, \tau_1 r) \subset E(\lambda, \tau_1 r )\subset \bigcup_{i=1}^{\infty} B_{5\rho_i}(y^i) \subset B_{\tau_2 r}.
$$
Thus
\[
\begin{split}
 &\int_{ E(K\lambda, \tau_1 r)}\phi(x,|Du|)\, dx \le \sum_{i=1}^{\infty} \int_{B_{5\rho_i}(y^i) \cap E(K\lambda, \tau_1 r)} \phi(x,|Du|)\, dx \\
 & \qquad\le c \epsilon \bigg( \int_{B_{\tau_2r}\cap\left\{ \phi(x,|Du|)>\frac{\lambda}{4} \right\}} \phi(x,|Du|) \,dx + \frac1{\delta} \int_{B_{\tau_2r} \cap\left\{\phi(x,|F|)>\frac{\delta\lambda}{4} \right\}} \phi(x,|F|) \,dx\bigg)
 \end{split}
\]
for some constant $c>0.$

We next apply a truncation argument to complete the proof.
 For $k \geq \lambda,$ let us define
$\phi_k(x, |Du|) := \min\left\{\phi(x,|Du|), k \right\},$
and consider the super-level set of $\phi_k $ 
\[
E_k ( \lambda, \rho):= \left\{ y \in B_{\rho} : \phi_k(y, |Du|)> \lambda\right\}\ \ \text{for } \lambda,\rho>0. 
\]
Note that $E_k ( \lambda,\rho) = E ( \lambda, \rho)$ if $\lambda\le k$ and $E_k(\lambda,\rho)=\emptyset$ if $\lambda>k$.
Hence we can continue the earlier estimate for any $\lambda\ge \alpha \lambda_0$,
\[
\begin{split}
\int_{ E_k(K\lambda, \tau_1 r)}\phi(x,|Du|)\, dx 
 \le c \epsilon \bigg( \int_{E_k\left(\frac{\lambda}{4},\tau_2r\right)}\phi(x, |Du|)\,dx + \frac1{\delta} \int_{B_{\tau_2r} \cap\left\{\phi(x,|F|)>\frac{\delta\lambda}{4} \right\}} \phi(x,|F|) \,dx\bigg).
 \end{split}
\]
Then
we multiply both sides by $\lambda^{s-2}$, $s\in (1,\infty)$, and integrate with respect to $\lambda$ over $(\alpha\lambda_0, \infty)$ to discover that
\[
\begin{split}
 I_0 &:=\int_{\alpha\lambda_0}^\infty \lambda^{s-2}\int_{ E_k(K\lambda, \tau_1 r)} \phi(x,|Du|) \, dxd\lambda\\
 & \leq c \epsilon \bigg( \int_{\alpha\lambda_0}^\infty \lambda^{s-2} \int_{E_k\left(\frac{\lambda}{4},\tau_2r\right)} \phi(x,|Du|) \,dxd\lambda \\
 &\qquad+\int_{\alpha\lambda_0}^\infty \lambda^{s-2} \int_{B_{\tau_2r}\cap\left\{\frac{\phi(x,|F|)}{\delta}>\frac{\lambda}{4} \right\}} \frac{\phi(x,|F|)}{\delta} \,dxd\lambda\bigg)\\
&=: c\epsilon (I_1+I_2).
 \end{split}
\]
By Fubini's theorem, we then derive that
\[
\begin{split}
I_0 
&= \frac{1}{s-1}\, \int_{E_k(K\alpha\lambda_0, \tau_1 r)}
\phi(x,|Du|) \bigg[\Big(\frac{\phi_k(x,|Du|)}{K}\Big)^{s-1} - (\alpha\lambda_0)^{s-1} \bigg] \, dx,
 \end{split}
\]
\[
\begin{split}
I_1&= \int_{E_k\left(\frac{\alpha\lambda_0}{4},\tau_2r\right)}\phi(x,|Du|) \bigg(\int_{\alpha\lambda_0}^{ 4\phi_k(x,|Du|)} \lambda^{s-2} \,d\lambda\bigg) dx \\
&\leq \frac{4^{s-1}}{s-1} \int_{B_{r_2 r}} \phi(x,|Du|) \phi_k(x,|Du|)^{s-1} \, dx,
 \end{split}
\]
and, similarly,
\[
I_2\leq \frac{4^{s-1}}{s-1} \int_{B_{\tau_2 r}} \left[ \frac{\phi(x,|F|)}{\delta}\right]^s \, dx.
\]
With these expressions for $I_0$, $I_1$, $I_2$, we continue our earlier estimate as
\[
\begin{split}
&\int_{B_{\tau_1r}} \phi(x,|Du|) \phi_k(x,|Du|)^{s-1} \, dx \\
&\le \int_{E_k(K\alpha\lambda_0, \tau_1 r)} \phi(x,|Du|) \phi_k(x,|Du|)^{s-1} \, dx + (K\alpha\lambda_0)^{s-1} \int_{B_{\tau_1 r}} \phi(x,|Du|) \, dx\\
 & \leq c_1\epsilon \int_{B_{\tau_2r}} \phi(x,|Du|)\phi_k(x,|Du|)^{s-1} \, dx + c (\alpha\lambda_0)^{s-1}\int_{B_{2 r}} \phi(x,|Du|) \,dx
\\
&\qquad +c \epsilon\int_{B_{2 r}} \left[\frac{\phi(x,|F|)}{\delta} \right]^s \, dx
\end{split}
\]
for some $c_1=c_1(n,p,q,L,L_1)\ge 0$. At this stage, we fix $\epsilon\in(0,1)$ such that $c_1\epsilon\le \frac12$, which also determines $\delta$ and $r$. From the definition of $\alpha$ we finally have 
\[
\begin{split}
\int_{B_{\tau_1 r}}\phi(x,|Du|) \phi_k(x,|Du|)^{s-1} \, dx
& \le \frac12 \int_{B_{\tau_2 r}} \phi(x,|Du|)\phi_k(x, |Du|)^{s-1} \, dx\\
&  +\frac{c\lambda_0^{s-1}}{(\tau_2-\tau_1)^{n(s-1)}}\int_{B_{2r}}\phi(x,|Du|)\,dx+ c \int_{B_{2 r}} \phi(x, |F|)^s \, dx .
\end{split}
\]
Then a standard iteration argument for $1\le \tau_1 < \tau_2 \le 2$ yields 
\[
\int_{B_{r}}\phi(x,|Du|) \phi_k(x,|Du|)^{s-1} \, dx \leq c\lambda_0^{s-1}\int_{B_{2r}}\phi(x,|Du|)\,dx+ c \int_{B_{2 r}} \phi(x,|F|)^s \, dx.
\]
Therefore, from Lebesgue's monotone convergence theorem, H\"older's inequality and Young's inequality together with the definition of $\lambda_0$ in \eqref{def_lambda0}, we derive that
\[
\begin{split}
\fint_{B_{r}}\phi(x,|Du|)^s\,dx &= \lim_{k\to\infty} \fint_{B_{r}}\phi(x,|Du|) \phi_k(x,|Du|)^{s-1} \, dx\\
& \le c\lambda_0^{s-1}\fint_{B_{2r}}\phi(x,|Du|)\,dx+ c \fint_{B_{2 r}} \phi(x,|F|)^s \, dx\\
& \le c\bigg(\fint_{B_{2r}}\phi(x,|Du|)\,dx\bigg)^s+ c \fint_{B_{2 r}} \phi(x,|F|)^s \, dx + c ,
\end{split}
\]
which together with the growth condition of growth functions implies the desired estimates. 
\end{proof}


\section{Continuity of the derivative}
\label{sect:proofC1}

We next prove Theorem~\ref{thm:C1}. In this section we always suppose that $A:\Omega\times \R^n\to\R^n$ satisfies the quasi-isotropic $(p,q)$-growth condition and 
$A^{(-1)}$ satisfies the \DMA{\gamma} condition for some $\gamma>2$.
Then $A^{(-1)}$ also satisfies the \VMA{} condition and, by Proposition~\ref{prop:DMA}, the \SA{} condition, and hence the \aone{} condition, with relevant constant $L$ depending on $n$, $p$, $q$, $L$, $L_1$ and $\theta_\gamma$. 
By Theorem~\ref{thm:CZ} with $F\equiv 0$, $\phi(\cdot,|Du|)\in L^{s}_{\loc}(\Omega)$
for any $s>1$.  Moreover, for $\Omega'\Subset\Omega$ there exists $R_1>0$ depending on $n$, $p$, $q$, $L$, $\gamma$, $\theta_\gamma$, $Du$ and $\Omega'$ and   such that  for any $B_{2r}\subset\Omega'$ with $r\le R_1$, 
the condition \eqref{eq:rrestriction1} holds,
\[
\varrho_{L^{\phi^{\gamma/(\gamma-2)}}(B_{2r})}(|Du|)=\int_{B_{2r}} \phi(x,|Du|)^{\frac{\gamma}{
\gamma-2}}\, dx \le 1,
\]
and, by Theorem~\ref{thm:CZ} with $F\equiv0$, $s=\frac{\gamma}{\gamma-2}$ and scaled balls,
\[
\bigg(\fint_{B_r} \phi(x,|Du|)^{\frac{\gamma}{\gamma-2}}\, dx\bigg)^{\frac{\gamma-2}{\gamma}} \le c \bigg(\fint_{B_{3r/2}} \phi(x,|Du|)\, dx+1\bigg) \le C^*_\gamma\, \tphi \bigg(\fint_{B_{2r}} |Du|\, dx+1\bigg),
\]
where the last inequality follows from Lemma~\ref{lem:GreverseDu}(2), with scaled balls, and the constant $C^*_\gamma>0$ depends on $n$, $p$, $q$, $L$ and $\gamma$.
Therefore, we obtain the necessary inequalities 
in the assumptions to apply Lemmas~\ref{lem:comparison1} and \ref{lem:comparison2}.

Let $u\in W^{1,\phi}_{\loc}(\Omega)$ be a weak solution to \eqref{maineq}. Fix 
$B_{2r}(x_0)\subset\Omega'$ with $r\le R_1$. 
For $\delta\in(0, \frac14)$ to be determined later and $k\in \N\cup\{0\}$, set
\[
\Theta_k :=\Theta(\delta^kr),
\quad
 B_k := B_{\delta^k r}(x_0), 
\quad\text{and}\quad
E_k := \fint_{B_k} |Du-(Du)_{B_k}|\,dx,
\]
where $\Theta$ is defined in Lemma~\ref{lem:comparison}. Suppose that 
\[
\delta\le \delta_1,
\]
where $\delta_1$ is given in Lemma~\ref{lem:theta-sum}. Then this lemma implies that 
\[
\sum_{k=0}^\infty \Theta_k = 
\sum_{k=0}^\infty \Theta(\delta^kr) 
\le 
\frac{c}{\delta^{(4n+1)/\gamma}\log\delta^{-1}}\int_0^r\Theta(\rho)\, \frac{d\rho}{\rho} 
< \infty.
\]
For each $k\in \N$, we construct a function $\tphi$ in the ball $B_k$ as in Section~\ref{sect:comparison} and 
let $\tu_k\in u+W^{1,\tphi}_0(B_k)$ be the weak solution to \eqref{eqv} in $B_k$. 
We will use the named contants
$C_0$, $C_1$, $C_2$ and $C_3$ from Lemmas~\ref{lem:GreverseDu}, \ref{lem:v_regularity},  
\ref{lem:comparison1} and \ref{lem:comparison2}. 
The following results hold assuming only the inequalities from these four results. 

Let us first record an estimate for $E_{k+1}$ with $k\ge j$. We integrate the inequality 
\[
|Du-(Du)_{B_{k+1}}| \le 
|Du-D\tu_j| + |D\tu_j-(D\tu_j)_{B_{k+1}}| + |(D\tu_j)_{B_{k+1}}-(Du)_{B_{k+1}}| 
\]
over $B_{k+1}$, then use Lemma~\ref{lem:v_regularity} with 
$(B_\varrho, B_\rho)\leftarrow(B_{k+1}, B_{k})$ and choose $\delta$ small so that 
$C_1 \delta^{\alpha} \le \frac1{32}$, to conclude that 
\begin{equation}\label{eq:Ek}
\begin{split}
E_{k+1} 
&\le 
\fint_{B_{k+1}} |D\tu_j-(D\tu_j)_{B_{k+1}}| \,dx + 2\fint_{B_{k+1}} |Du-D\tu_j| \,dx\\
&\le
C_1 \delta^{\alpha}\fint_{B_{k}} |D\tu_j-(D\tu_j)_{B_{k}}| \,dx
+ 2\delta^{-n}\fint_{B_{k}} |Du-D\tu_j| \,dx \\
&\le
2C_1 \delta^{\alpha}\fint_{B_{k}} |D\tu_j-(Du)_{B_{k}}| \,dx
+ 2\delta^{-n}\fint_{B_{k}} |Du-D\tu_j| \,dx \\
&\le 
\tfrac1{16} E_k
+ 4\delta^{-n} \fint_{B_{k}} |Du-D\tu_j| \,dx.
\end{split}
\end{equation}

\begin{lemma}\label{lem:E-iteration}
In the setting fixed in the beginning of the section, let $\delta\in(0,\delta_1]$, $\epsilon \in (0,\frac14)$ and $k\in\N$ satisfy 
\[
\max\{C_2 \delta^{-n}\Theta_k^{1/q}, 32C_0C_1 \delta^{\alpha} \} \le \epsilon,
\]
where $\alpha\in(0,1)$ is from Lemma~\ref{lem:v_regularity}.
If
\[
\fint_{2B_{k}}|Du|\,dx \le \lambda 
\quad\text{and}\quad
\fint_{2B_{k+1}}|Du|\,dx \ge \epsilon\lambda, 
\]
for some $\lambda \ge 1$, then 
\[
E_{k+2}
\le 
\tfrac1{16} E_{k+1} + 
4^q C_3 \epsilon^{-(q-1)} \delta^{-2n} \Theta_k\lambda.
\]
\end{lemma}
\begin{proof}
We use the second inequality in the assumption and Lemma~\ref{lem:comparison1} 
in $B_{k}$ to conclude that
\[
\epsilon \lambda 
\le 
\fint_{2B_{k+1}} |Du-D\tu_k|\,dx + \fint_{2B_{k+1}} |D\tu_k|\,dx 
\le 
2^{1-n}C_2 \delta^{-n}\Theta_k^{1/q} \lambda + \fint_{2B_{k+1}} |D\tu_k|\,dx.
\]
By the assumption $C_2 \delta^{-n}\Theta_k^{1/q}\le \epsilon$, 
the first term on the right is at most 
$\frac{\epsilon \lambda}{2}$ and can be absorbed into the left-hand side.
By Lemma~\ref{lem:v_regularity} with $(B_\rho, B_r)\leftarrow(2B_{k+1}, B_k)$, we have
\[
\inf_{2B_{k+1}} |D\tu_k| 
\ge 
\fint_{2B_{k+1}} |D\tu_k|\,dx - \osc_{2B_{k+1}}|D\tu_k| 
\ge 
\frac{\epsilon \lambda}{2} - 2^{1+\alpha}C_1 \delta^{\alpha } \fint_{B_k} |D\tu_k|\,dx 
\ge 
\frac{\epsilon \lambda}{4},
\]
where in the last step we used Lemma~\ref{lem:GreverseDu}(3) with $\tu=\tu_k$ and the assumption $32C_0C_1 \delta^{\alpha} \le \epsilon$. 
Therefore, applying Lemma~\ref{lem:comparison2} with 
$(\tu,\widetilde C, B_r, B_\rho)\leftarrow(\tu_k,\frac{4}{\epsilon}, B_k, B_{k+1})$, we estimate
\[
\fint_{B_{k+1}} |Du-D\tu_k|\, dx 
\le 
C_3 (\tfrac4\epsilon)^{q-1} \delta^{-n}\Theta_k \lambda.
\]
Note that $C_1 \delta^{\alpha} \le \frac1{32}$ follows from our assumption so we can use 
\eqref{eq:Ek} with $(k, j)\leftarrow(k+1, k)$ and 
the previous estimate
to conclude that 
\[\begin{split}
E_{k+2} 
&\le 
\tfrac1{16}E_{k+1}
+ 4\delta^{-n} \fint_{B_{k+1}} |Du-D\tu_k| \,dx 
\le 
\tfrac1{16} E_{k+1} + 
4^q C_3 \epsilon^{-(q-1)} \delta^{-2n} \Theta_k\lambda. \qedhere 
\end{split}\]
\end{proof}

%


\begin{lemma}\label{lem:induction}
In the setting fixed in the section, let $\delta\in (0,\delta_1]$, $\epsilon_0:=2^{-n-2}$, $\epsilon\in(0,\epsilon_0]$ and $j\in\N$ satisfy the conditions
\[
\max\{64 C_2 \delta^{-2n}\sup_{k\ge j}\Theta_k^{1/q}, 32C_0C_1 \delta^{\alpha}\} \le \epsilon_0\epsilon
\quad\text{and}\quad
\sum_{k=j}^{\infty}\Theta_{k}\le \frac{\delta^{3n}\epsilon_0\epsilon^{q}}{8 \cdot 4^q C_3}\,.
\]
For $m\ge j+1$, we assume that 
\begin{equation}\label{eq:inductionAss1} 
\frac1{\delta^n\epsilon} E_{j} + 
\fint_{2B_{j}}|Du|\,dx \le 2\epsilon_0\lambda
\end{equation}
and 
\begin{equation}\label{eq:inductionAss2}
\fint_{2B_k}|Du|\,dx \ge \epsilon\lambda \ \ \text{for every }\ k\in \{j+1,\ldots, m-1\} \ \ \text{when }\ m\ge j+2,
\end{equation}
for some $\lambda \ge 1$. Then 
\[
E_m\le \frac{\delta^n\epsilon_0\epsilon}{2}\lambda,\quad
\fint_{B_{m+1}}|Du - (Du)_{B_{j}}|\,dx \le 3\epsilon_0\epsilon \lambda 
\quad\text{and}\quad 
\fint_{2B_m}|Du|\,dx \le \lambda.
\]
\end{lemma}
\begin{proof} 
We first note that \eqref{eq:Ek} with $k=j$, \eqref{eq:inductionAss1} and Lemma~\ref{lem:comparison1} in $B_j$ imply that
\[
E_{j+1} 
\le 
\tfrac1{16} E_j
+ 4\delta^{-n} \fint_{B_{j}} |Du-D\tu_j| \,dx
\le 
\frac{\delta^n \epsilon_0\epsilon}{8} \lambda + 
4C_2 \delta^{-n} \Theta_j^{1/q} \bigg(\fint_{2B_j}|Du|\, dx+1\bigg).
\]
By assumption $4C_2 \delta^{-n} \Theta_j^{1/q} \le 
\frac{1}{16}\delta^n\epsilon_0\epsilon$ and $\fint_{2B_j}|Du|\, dx+1\le \frac32 \lambda$ 
so that $E_{j+1}\le \frac14 \delta^n \epsilon_0\epsilon \lambda$. 
Moreover, by the assumption \eqref{eq:inductionAss1}, 
\begin{align*}
\fint_{2B_{j+1}}|Du|\,dx 
& \le \fint_{2B_{j+1}}|Du-(Du)_{B_{j}}|\,dx + |(Du)_{B_{j}}| 
\le (2\delta)^{-n}E_j + 2^{n}\fint_{2B_{j}}|Du|\,dx \\
&\le 2^n \cdot 2\epsilon_0\lambda \le \lambda. 
\end{align*}
Using also the previous estimate for $E_{j+1}$, we further obtain that 
\begin{align*}
\fint_{B_{j+2}}|Du - (Du)_{B_{j}}|\,dx 
&\le
\fint_{B_{j+2}}|Du - (Du)_{B_{j+1}}|\,dx + |(Du)_{B_{j+1}} - (Du)_{B_{j}}| \\
&\le 
\delta^{-n} (E_{j+1}+E_j) 
\le 
(\tfrac14 + 2) \epsilon_0 \lambda 
\le
3\epsilon_0 \lambda.
\end{align*}
Thus we have proved all the claims in the case $m=j+1$.

We prove the claim by induction when $m\ge j+2$; specifically, we assume that 
the claim holds up to $m-1$ and show it for $m$.
By the induction assumption $\fint_{2B_k}|Du|\, dx \le \lambda$ when $k\le m-1$. 
The assumption \eqref{eq:inductionAss2} gives $\fint_{2B_{k+1}}|Du|\, dx \ge \epsilon\lambda$ for 
$k\in \{j,\ldots, m-2\}$. Thus we can use Lemma~\ref{lem:E-iteration} to conclude that 
\[
2E_{k+2} - E_{k+1} 
\le 
c' \Theta_{k}\lambda
\]
for $k\in \{j, \ldots, m-2\}$ and $c':=2\cdot 4^q C_3 \epsilon^{-(q-1)} \delta^{-2n}$. 
Therefore
\begin{equation}\label{lem:induction_pf1}
\sum_{k=j+1}^{m} E_{k}
=
\sum_{k=j}^{m-2} (2E_{k+2} - E_{k+1}) + 2E_{j+1} - E_{m}
\le 
c' \lambda \sum_{k=j}^{m-2}\Theta_{k} + 2E_{j+1} \le \frac12 \delta^n \epsilon_0\epsilon \lambda
\end{equation}
since $\sum_{k=j}^{\infty}\Theta_{k}\le \frac{\delta^n\epsilon_0\epsilon}{4c'}$ by assumption
and $E_{j+1} \le \frac{\delta^n\epsilon_0\epsilon}4 \lambda$ as shown above. 
In particular, the first claim $E_m\le \frac12\delta^n\epsilon_0\epsilon\lambda$ is proved. 
We write
\[
Du-(Du)_{B_{j}} = 
[Du - (Du)_{B_m}] + [(Du)_{B_m} - (Du)_{B_{m-1}}] + \cdots 
+ [(Du)_{B_{j+1}} - (Du)_{B_{j}}].
\]
Integrating this over $B_{m+1}$ and using \eqref{eq:inductionAss1} and \eqref{lem:induction_pf1} to estimate $E_k$, we observe that
\begin{align*}
\fint_{B_{m+1}}|Du - (Du)_{B_{j}}|\,dx 
&\le
\sum_{k=j}^{m} \fint_{B_{k+1}} |Du-(Du)_{B_{k}}|\, dx 
 \le
\delta^{-n}\sum_{k=j}^{m} E_k 
\le 
3\epsilon_0\epsilon \lambda.
\end{align*}
This proves the second claim.
Similarly, we prove the inequality $\fint_{2B_m}|Du - (Du)_{B_{j}}|\,dx\le (2\delta)^{-n} \sum_{k=j}^{m-1}E_k \le \frac{3\epsilon_0\epsilon}{2^{n}} \lambda$.
Using this and the estimate $\fint_{2B_j}|Du|\,dx\le \frac32\epsilon_0 \lambda = 3\cdot 2^{-n-3}\lambda$, we obtain
\[
\fint_{2B_m}|Du|\,dx 
\le
\fint_{2B_m}|Du - (Du)_{B_{j}}|\,dx + 2^n\fint_{2B_j}|Du|\,dx
\le
\frac{3\epsilon_0\epsilon}{2^n} \lambda + \frac18\lambda 
\le \lambda.
\]
This concludes the proof of the third claim in the induction step, 
so the claim holds for any number of steps. 
\end{proof}

\begin{proposition}\label{prop:bounded}
In the setting fixed in the beginning of the section, 
the gradient $Du$ is locally essentially bounded. 
Moreover, 
\[
\|Du\|_{L^\infty(B_{r})} \le c \left(\fint_{B_{2r}}|Du|\,dx +1\right)
\]
for some $c=c(n,p,q,\gamma, L,L_1,L_{\gamma},\theta_\gamma)\ge 1$ and any 
$B_{2r}\subset \Omega'$ with $r\in(0,R_1]$. 
\end{proposition}
\begin{proof}
Let $\epsilon:=\epsilon_0=2^{-n-2}$ and choose $\delta:=\min\{\delta_1,(\frac{\epsilon_0^2}{32C_0C_1})^{1/\alpha}\}$. 
Let $j_0$ satisfy the conditions of Lemma~\ref{lem:induction} with $j=j_0$; 
such $j_0$ exists since the series $\sum \Theta_k$ is convergent.

Now we fix a Lebesgue point $x_0\in B_r$ of $Du$ and consider balls $B_k:=B_{\delta^kr}(x_0)$, $k\in\mathbb N$. For $k\ge j_0$,
we define
\[
F_k := \frac1{\delta^n \epsilon_0}E_{k} + \fint_{2B_{k}} |Du|\, dx
\quad\text{and}\quad
\lambda:= \tfrac1{2\epsilon_0} F_{j_0} +1.
\]
Since $\lambda \le F_{j_0}\lesssim \delta^{-j_0} \fint_{B_{2r}}|Du|\, dx$ and 
both $\delta$ and $j_0$ depend only on the parameters and can be included in the 
constant $c$, we complete the proof by showing that $|Du(x_0)|\le \lambda$.

Since $x_0$ is a Lebesgue point, 
\[
|Du(x_0)| \le \liminf_{k\to \infty} \fint_{2B_k} |Du|\, dx.
\]
The claim follows if $F_k \le 2\epsilon_0\lambda$ 
for infinitely many $k$. If this is not the case we let $j\ge j_0$ be the largest 
index with $F_j \le 2\epsilon_0\lambda$; note that such index exists as $F_{j_0} \le 2\epsilon_0\lambda$ 
by the choice of $\lambda$. Then $F_k \ge 2\epsilon_0 \lambda$ for every $k>j$. 
Now Lemma~\ref{lem:induction} with $m=j+1$ implies that $E_{j+1}\le \frac12\delta^n\epsilon_0^2 \lambda$. Therefore, 
\[
\fint_{2B_{j+1}} |Du|\, dx 
= 
F_{j+1} - \delta^{-n} \epsilon_0^{-1} E_{j+1}
\ge 
2\epsilon_0 \lambda - \tfrac12\epsilon_0 \lambda
\ge 
\epsilon_0 \lambda.
\]
From this we conclude that \eqref{eq:inductionAss2} is satisfied for 
$m=j+2$, which by the lemma in turn implies that 
$E_{j+2}\le \frac12\delta^n\epsilon_0^2 \lambda$. 
This in turn implies $E_{j+3}\le \frac12\delta^n\epsilon_0\lambda$, 
and so on.
Thus \eqref{eq:inductionAss2} holds for every $k>j$ and so 
Lemma~\ref{lem:induction} implies that $\fint_{2B_m} |Du|\, dx \le \lambda$ 
for every $m\ge j$, so that $|Du(x_0)| \le \lambda$. 
\end{proof}

Now, we are ready to prove the continuity of the derivative $Du$.

\begin{proof}[Proof of Theorem~\ref{thm:C1}]
We continue in the setting fixed in the beginning of the section and note that 
Proposition~\ref{prop:bounded} implies that $Du$ is bounded in $\Omega'$.
Let $\lambda:=\epsilon_0^{-1}\|Du\|_{L^\infty(\Omega')}+1$. For $\epsilon\in (0,\epsilon_0]$, choose 
$\delta:=\min\{\delta_1,(\frac{\epsilon_0\epsilon}{32C_0C_1})^{1/\alpha}\}$ and then $j_0$ so large that the conditions of Lemma~\ref{lem:induction} are satisfied when $j=j_0$. 
We observe that $E_{j_0}\le \lambda$ by the definition of $\lambda$ and 
$E_{k+1} \le \tfrac1{16} E_k+ \tfrac18 \delta^n \epsilon_0\epsilon \lambda$ for any $k\ge j_0$ (see 
beginning of the proof of Lemma~\ref{lem:comparison1}). 
Thus we may assume by increasing $j_0$ if necessary that 
$E_j\le \delta^n \epsilon_0 \epsilon\lambda $ for all $j\ge j_0$.
Then assumption \eqref{eq:inductionAss1} of Lemma~\ref{lem:induction} holds for any $j\ge j_0$.

Fix $\Omega''\Subset\Omega'$ and assume by increasing $j_0$ if necessary that 
$2B_{j_0}(x)\subset \Omega'$ for every $x\in \Omega''$. 
Consider Lebesgue points $x, y \in \Omega''$ of $Du$  
so close to one another that $|B_j(x) \setminus B_j(y)| \le \epsilon |B_j|$ for some 
$j\ge j_0$. Suppose that
\[
\big|(Du)_{B_j} - (Du)_{B_{m+1}}\big| \le (6+2^{n+1})\epsilon\lambda 
\]
for all $m\ge j+2$ (this is proved below). Then 
\begin{align*}
|Du(x)-Du(y)| 
&\le 
\lim_{m\to\infty} \big|(Du)_{B_j(x)} - (Du)_{B_{m+1}(x)}\big|
+\lim_{m\to\infty} \big|(Du)_{B_j(y)} - (Du)_{B_{m+1}(y)}\big| \\
&\qquad + \big|(Du)_{B_j(x)}-(Du)_{B_j(y)}\big| \\
&\le (6+2^{n+1})\epsilon\lambda + 2\,|B_j(x) \setminus B_j(y)|\, |B_j|^{-1} \|Du\|_{L^\infty(\Omega')} \le 2^{n+2}\epsilon\lambda.
\end{align*}
This implies the uniform continuity of $Du$ in $\Omega''$.

It remains to prove the inequality $\big|(Du)_{B_j} - (Du)_{B_{m+1}}\big| \le (6+2^{n+1})\epsilon\lambda $. Let us set
\[
K:=\big\{k\in\N: F_k < 2\epsilon\lambda\big\},
\]
where $F_k$ is defined in the proof of Proposition~\ref{prop:bounded}.
If $K\cap [j+1, m-1]=\emptyset$, then 
\[
\fint_{2B_{k}} |Du|\, dx 
= 
F_{k} - \delta^{-n} \epsilon_0^{-1} E_{k}
\ge 
2\epsilon \lambda - \epsilon \lambda
\ge 
\epsilon \lambda 
\quad \text{for all }\ j+1 \le k \le m-1.
\]
Therefore, the second assumption \eqref{eq:inductionAss2} of Lemma~\ref{lem:induction}
holds and so, from the second estimate in Lemma~\ref{lem:induction}, the desired inequality follows. 
Otherwise, let $j'=\max(K\cap [j+1, m-1])$ and 
$m'=\min(K\cap [j+1, m-1])$. Then applying Lemma~\ref{lem:induction}
when $m=m'-1>j$, we conclude that 
\[
|(Du)_{B_j} - (Du)_{B_{m'}}| \le 3\epsilon_0 \epsilon \lambda ,
\]
which together with the fact that $m'\in K$ implies
\[
|(Du)_{B_j}| 
\le 
|(Du)_{B_j}-(Du)_{B_{m'}} | + |(Du)_{B_{m'}}| 
\le 
3\epsilon_0\epsilon\lambda + F_{m'}
\le 
(3+2^n)\epsilon\lambda.
\]
Similarly, applying the second estimate in Lemma~\ref{lem:induction} when $j=j'$, we also 
obtain $|(Du)_{B_{j'}} - (Du)_{B_m}| \le \epsilon_0 \epsilon \lambda$ and so $|(Du)_{B_m}| \le (3+2^n)\epsilon\lambda$. Therefore, $|(Du)_{B_j} - (Du)_{B_m}| \le |(Du)_{B_j}| + |(Du)_{B_m}| \le(6+2^{n+1})\epsilon\lambda$, as claimed.
\end{proof}

%
%
%

\section*{Acknowledgment}

M. Lee was supported by NRF grant funded by MSIT (NRF-2022R1F1A1063032). J. Ok was supported by NRF grant funded by MSIT (NRF-2022R1C1C1004523).
\bibliographystyle{amsplain}

\end{document}